\documentclass[a4paper,12pt]{article}
\usepackage{amsthm}
\usepackage{hyperref}
\usepackage[utf8]{inputenc}
\usepackage{extpfeil}
\usepackage[
backend=biber,
style=numeric
]{biblatex}
\DeclareNameAlias{default}{last-first}
\usepackage{subcaption}
\usepackage{geometry}
\usepackage{lipsum}
\usepackage{tikz}
\usetikzlibrary{decorations.pathreplacing,shapes.misc}
\usepackage{pgfplots}
\usepgfplotslibrary{fillbetween}
\pgfplotsset{compat=1.11}
\usepackage{mathtools}
\usepackage{changepage}
\usepackage{theoremref}
\usepackage{xfrac}
\newlength{\abstractwidth}
\renewcommand{\thanks}[1]{\footnote{#1}} 

\newcommand{\be}{\begin{equation}}
\newcommand{\bea}{\begin{eqnarray}}
\newcommand{\eea}{\end{eqnarray}} \newcommand{\ee}{\end{equation}}
 
 \def\ba{\begin{eqnarray}}
\def\ea{\end{eqnarray}}

\def\[{{\bf [}}
\def\]{{\bf ]}}

\begin{document}
\newtheorem{theorem}{Theorem} [section]
\newtheorem{proposition}[theorem]{Proposition} 
\newtheorem{lemma}[theorem]{Lemma} 
\newtheorem{corollary}[theorem]{Corollary} 
\newtheorem{definition}[theorem]{Definition} 
\newtheorem{conjecture}[theorem]{Conjecture} 
\newtheorem{example}[theorem]{Example} 
\newtheorem{claim}[theorem]{Claim} 
\newtheorem{remark}[theorem]{Remark} 
\newtheorem*{proposition*}{Proposition}

\begin{centering}
 
\textup{\LARGE\bf Inverse Mean Curvature Flow of \\ \vspace{0.3cm} Rotationally Symmetric Hypersurfaces}

\vspace{10 mm}

\textnormal{\large Brian Harvie} \\

\vspace{.5 in}
\begin{abstract}
{\small We prove that the Inverse Mean Curvature Flow of a non-star-shaped, mean-convex embedded sphere in $\mathbb{R}^{n+1}$ with symmetry about an axis and sufficiently long, thick necks exists for all time and homothetically converges to a round sphere as $t \rightarrow \infty$. Our approach is based on a localized version of the parabolic maximum principle.

We also present two applications of this result. The first is an extension of the Minkowski inequality to the corresponding non-star-shaped, mean-convex domains in $\mathbb{R}^{n+1}$. The second is a connection between IMCF and minimal surface theory. Based on previous work by Meeks and Yau in \cite{Meeks1982TheEO} and using foliations by IMCF, we establish embeddedness of the solution to Plateau's problem and a finiteness property of stable immersed minimal disks for certain Jordan curves in $\mathbb{R}^{3}$.
}

\end{abstract}


\end{centering}

\section{Introduction}
\hspace{0.5cm} 


The dynamical stability of the round sphere under extrinsic geometric flows is both crucial to applications in geometry and topology and a fascinating study in its own right. Work on the convex stability of round spheres in Mean Curvature Flow (MCF) goes back to papers by Gage-Hamilton in \cite{gage1986} and Huisken in \cite{huisken1984}. On the other hand, the most prominent expanding extrinsic flow is the Inverse Mean Curvature Flow (IMCF). Given a closed, oriented $n$-dimensional smooth manifold $N$, a one-parameter family of $C^{\infty}$ immersions $F: N \times [0,T) \rightarrow \mathbb{R}^{n+1}$ moves by IMCF if

\begin{eqnarray} \label{IMCF}
    \frac{\partial F}{\partial t} (x,t) &=& \frac{1}{H} \nu (x,t), \hspace{1cm} (x,t) \in N \times [0,T),
\end{eqnarray}
where $\nu$ is the outward-pointing unit normal and $H > 0$ the mean curvature of $N_{t}=F_{t}(N)$. Gerhardt showed in \cite{claus} that IMCF flows star-shaped initial data $N_{0}=F_{0}(\mathbb{S}^{n}) \subset \mathbb{R}^{n+1}$ into round spheres after rescaling in the sense of $C^{\infty}$ convergence, see also Urbas in \cite{Urbas} and Huisken and Ilmanen in \cite{huisken2}. This result is striking given the highly nonlinear profile of \eqref{IMCF}, and so we are interested in understanding the dynamical stability of IMCF in the non-star-shaped regime.

This paper studies the evolution of mean-convex spheres with rotational symmetry about an axis under IMCF. Rotationally symmetric surfaces have been a rich source of examples of non-trivial singularities in the case of MCF, see \cite{angenent}, \cite{degenerate}, \cite{zhou}, \cite{grayson}. As with star-shaped IMCF, the main difficulty in this setting is that the evolution equations are fully nonlinear. In addition, many quantities behave differently near the axis of rotation than they do away from it, and so we need a way to localize estimates to different regions of the evolving surface. Ultimately, we show like with star-shaped IMCF that any $H>0$ rotationally symmetric embedded sphere $N_{0}$ obeying an admissibility condition admits a solution $\{N_{t}\}_{0 \leq t < T_{\max}}$ to IMCF which exists forever, remains embedded, and rapidly converges to spheres after scaling.

\begin{theorem} [Rotationally Symmetric Stability of Round Spheres in IMCF] \thlabel{long_time}
Let $F_{0}: \mathbb{S}^{n} \rightarrow \mathbb{R}^{n+1}$ be a $C^{\infty}$, $H>0$ rotationally symmetric embedding that is admissible in the sense of \thref{admissible}. Then $T_{\max}=+\infty$ for the corresponding maximal solution $F: \mathbb{S}^{n} \times [0,T_{\max}) \rightarrow \mathbb{R}^{n+1}$ of \eqref{IMCF}. Furthermore, $N_{t}=F_{t}(\mathbb{S}^{n})$ is embedded for all $t \in [0,\infty)$ and star-shaped for $t \geq n \text{log}(R^{-1} \text{diam} (N_{0}))$, where $R$ is the radius of the largest ball $N_{0}$ encloses. As a consequence, there is an $x_{0} \in \mathbb{R}^{n+1}$ and $r>0$ such that the re-scaled surfaces $\tilde{N}_{t}=e^{-\frac{t}{n}} N_{t}$ converge in $C^{\infty}$ topology to $\partial B_{r}(x_{0})$ (in the sense of smooth convergence of some choice of immersions) as $t \rightarrow \infty$.
\end{theorem}

\thref{admissible} includes a $C^{1}$ condition on a rotationally symmetric surface $N_{0}$ that controls the shape of its ``necks''. Specifically, it ensures that each neck of the surface has a certain minimal thickness and length. We give an example of a non-star-shaped $N_{0}$ satisfying \thref{admissible} in Appendix A2.

One application of this theorem is to geometric inequalities. The Minkowski inequality is a lower bound on the $L^{1}$ norm of the mean curvature $H$ of a closed, convex hypersurface $N_{0} \subset \mathbb{R}^{n+1}$ in terms of its area. Because of its connection with the isoperimetric problem among others, one would like to understand which non-convex bodies this inequality extends to. Guan and Li first observed in \cite{Guan2009TheQI} that the Minkowski inequality holds on a mean-convex embedded hypersurface $N_{0}$ whenever $N_{0}$ admits a long-time solution $N_{t}$ to IMCF that becomes asymptotically round. Therefore, \thref{long_time} leads to the following corollary for admissible rotationally symmetric hypersurfaces.

\begin{corollary}[Minkowski Inequality] \thlabel{minkowski}
Let $N_{0}=F_{0}(\mathbb{S}^{n}) \subset \mathbb{R}^{n+1}$ be an embedded, $H>0$ rotationally symmetric hypersurface which is admissible in the sense of \thref{admissible}. Then the Minkowski inequality

\begin{equation}
     \int_{N_{0}} H d\mu \geq n |\partial B_{1}(0)|^{\frac{1}{n}} |N_{0}|^{\frac{n-1}{n}}
\end{equation}
holds on $N_{0}$, with equality if and only if $N_{0}$ is round.
\end{corollary}

A second application of this theorem relates to minimal surface theory. Two longstanding and closely-related questions in this field concern $(1)$ the number of stable minimal disks that span a given Jordan curve $\gamma \subset \mathbb{R}^{3}$, and $(2)$ whether or not the least-area disk spanning $\gamma$ is embedded. In some cases, $\gamma$ may bound infinitely many stable minimal disks, and its least-area disk may self-intersect, see \cite{morgan} and \cite{Coskunuzer_2012}, respectively. Some other results on each of these problems can be found in, e.g., \cite{Almgren1979ExistenceOE}, \cite{Ekholm2002EmbeddednessOM}, \cite{rossman}, \cite{beeson}, \cite{hardt_simon}, and \cite{Beeson2006ARJ}.

These questions are more tractable when considering a Jordan curve $\gamma$ that lies on the mean-convex boundary $\partial E_{0}$ of a bounded domain $E_{0} \subset \mathbb{R}^{3}$. In their landmark paper \cite{Meeks1982TheEO}, W.H. Meeks and S.T. Yau showed for any such $\gamma$ that there is an embedded disk $D \subset E_{0}$ which spans $\gamma$ and minimizes area among all other such disks in $E_{0}$. They also showed that the number of stable minimal disks in $E_{0}$ that span $\gamma$ and satisfy a uniform area bound is finite. In general, neither of these results apply to all disks in $\mathbb{R}^{3}$ that span this $\gamma$. However, we will show that all immersed minimal surfaces bounded by $\gamma$ lie within $E_{0}$ whenever the surface $N_{0}=\partial E_{0}$ admits a global, embedded solution to \eqref{IMCF}. From this, we obtain the following theorem.

\begin{theorem}[Embeddedness and Finiteness of Area-Minimizers] \thlabel{embeddedness}
Let $N_{0}=F_{0}(\mathbb{S}^{2}) \subset \mathbb{R}^{3}$ be an $H>0$ embedded surface such that the maximal solution $N_{t}$ to \eqref{IMCF} exists and remains embedded for all time, and let $\gamma \subset N_{0}$ be a Jordan curve. Then the least-area disk $D$ spanning $\gamma$ in $\mathbb{R}^{3}$ is embedded. Furthermore, if $\gamma$ is a $C^{4,\alpha}$ Jordan curve, then for any $k \in \mathbb{R}$ it bounds only finitely many stable immersed minimal disks with areas less than $k$. 

In particular, these properties hold whenever $\gamma$ lies on an $H>0$ surface $N_{0}$ which is either star-shaped or is rotationally symmetric and admissible in the sense of \thref{admissible}.
\end{theorem}

The paper is organized as follows: in section 2, we discuss elementary properties of rotationally symmetric immersions. In section 3, we show that each immersion $F_{t}$ in \eqref{IMCF} is an embedding if $N^{n}=\mathbb{S}^{n}$ and $F_{0}$ is a rotationally symmetric embedding. The subsequent sections focus on establishing long-time existence for a rotationally symmetric solution. Our approach here is a localized or ``non-cylindrical" version of the parabolic maximum principle that we present in section 4. We also define domains in $\mathbb{S}^{n} \times [0,T)$ that correspond to different regions of the flow surfaces $N_{t}$ in this section.

After deriving evolution equations and $C^{0}$ estimates in sections 5 and 6, we apply the non-cylindrical maximum principle over each of these domains in sections 7 and 8. First, in section 7, we inspect the ``bridge'' region of the surface which lies away from the axis of rotation. We show that the flow speed $H^{-1}$ can be controlled over this region via a sharp gradient-like estimate which is obtained using the $C^{1}$ assumption of \thref{admissible}. In section 8, we obtain a bound on $H^{-1}$ over the ``cap" regions which intersect the axis of rotation. Combining these estimates allows us complete the proof of \thref{long_time} in section 9.

Section 10 is dedicated to the applications of \thref{long_time}. \thref{minkowski} more or less follows immediately. The proof of \thref{embeddedness} uses the fact proved in \cite{me} that an embedded solution of \eqref{IMCF} foliates its image in $\mathbb{R}^{n+1}$. A mean-convex foliation of the region exterior to $N_{0}$ controls the position of immersed minimal disks via a comparison principle.

The version of the parabolic maximum principle used in this paper is non-standard, as the underlying domain is non-cylindrical in a sense described in section 4. We justify this modification in the Appendix A.1. In Appendix A.2, we construct rotationally symmetric initial data $N_{0}$ that satisfies the hypothesises of \thref{long_time} and is not star-shaped.

\section*{Acknowledgements}
I would like to thank Professors Joel Hass and Adam Jacob of the University of California, Davis and Professor Mao-Pei Tsui for National Taiwan University for many helpful discussions. I would also like to thank the University of California, Davis Department of Mathematics and the National Center for Theoretical Sciences, Mathematics Division at National Taiwan University for their financial support throughout my graduate
 and postdoctoral studies.

\section{Properties of Rotationally Symmetric Immersions}
\hspace{0.5cm} Consider a $C^{\infty}$ orientation-preserving immersion $F_{0}: N^{n} \rightarrow \mathbb{R}^{n+1}$ with an image $N_{0}=F_{0}(N)$ that is fixed by the set of all rotations about some axis in $\mathbb{R}^{n+1}$. We choose a Cartesian coordinate system of the ambient space so that the axis of symmetry is the $x_{1}$-axis, which we will denote as a set by $X_{1} \subset \mathbb{R}^{n+1}$. Two important ambient vector fields for extracting information about $N_{0}$ are

\begin{eqnarray}
    \hat{w}(x) &=& \frac{(0, x_{2}, \dots, x_{n+1})}{(x_{2}^{2} + \dots + x_{n+1}^{2})^{\frac{1}{2}}}, \hspace{0.5cm} x \in \mathbb{R}^{n+1} \setminus X_{1} \label{w} \\
    \hat{e}_{1}(x) &=& (1,0,0, \dots, 0), \hspace{0.5cm} x \in \mathbb{R}^{n+1}. \label{e1}
\end{eqnarray}
The distance to $X_{1}$ or ``height" $u$ and the $x_{1}$-coordinate $\tilde{u}$ of the image point of $x \in N$ are
\begin{eqnarray}
    u(x) &=& \langle w, \vec{F_{0}} (x) \rangle, \\
    \tilde{u}(x) &=& \langle e_{1}, \vec{F_{0}}(x) \rangle. \label{x_1}
\end{eqnarray}
The Killing vector fields corresponding to rotation about $X_{1}$ are all orthogonal to both $\hat{w}$ and $\hat{e}_{1}$ on $\mathbb{R}^{n+1} \setminus X_{1}$ and are necessarily tangent to $N_{0}$, so for $F_{0}(x) \not \in X_{1}$ any choice of unit normal $\nu$ on $N_{0}$ must lie in $\text{span}\{ e_{1}, w \} \subset T_{F_{0}(x)}\mathbb{R}^{n+1}$. From this, one can easily see that the unit vector

\begin{equation} \label{v_1}
    \hat{v}_{1} = \langle \nu, w \rangle \hat{e}_{1} - \langle \nu, e_{1} \rangle \hat{w}
\end{equation}
in $\text{span}\{ e_{1}, w \}$ is tangent to $N_{0}$ at $F_{0}(x)$. If we include $\hat{v}_{1}$ in an orthonormal basis $\hat{v}_{1}, \dots, \hat{v}_{n}$ of $T_{F_{0}(x)}N_{0}$, the vectors $\hat{v}_{2}, \dots, \hat{v}_{n}$ each correspond to directions of rotation about $X_{1}$. In particular, 
\begin{equation*}
    \hat{v}_{i}(\langle \nu, w \rangle)=\hat{v}_{i}(\langle \nu, e_{1} \rangle) = \hat{v}_{i}((x_{2}^{2} + \dots x_{n}^{2}))= 0
\end{equation*}
 for $i > 1$, and the Euclidean covariant derivatives of $\hat{v}_{1}$ with respect to these $\hat{v}_{i}$ are

\begin{eqnarray}
    \nabla_{\hat{v}_{i}} \hat{v}_{1} &=& -\langle \nu, e_{1} \rangle \nabla_{\hat{v}_{i}} \hat{w} = - \langle \nu, e_{1} \rangle \frac{\hat{v}_{i}(x_{j})}{(x_{2}^2 + \dots x_{n}^{2})^{\frac{1}{2}}} \partial_{x_{j}} \label{v_i} \\
    &=& - \langle \nu, e_{1} \rangle \frac{v^{j}_{i}}{(x_{2}^2 + \dots x_{n}^{2})^{\frac{1}{2}}} \partial_{x_{j}} = - \frac{\langle \nu, e_{1} \rangle}{u} \hat{v}_{i}. \hspace{0.5cm} (i,j > 1) \nonumber
\end{eqnarray}

Thus $\hat{v}_{1}$ is an eigenvector of the second fundamental form $A(X,Y)=-\langle \nabla_{X} Y, \nu \rangle$ of $N_{0}$ at $F_{0}(x)$-- we call its principal curvature $k(x)$. Every other $\hat{v}_{i}$ in this basis is tangent to a circle of radius $u(x)$ centered about $X_{1}$ and lieing on $N_{0}$. So the acceleration vector of $\hat{v}_{i}$ in $\mathbb{R}^{n+1}$ for $i >1$ is

\begin{equation} \label{other_v_i}
    \nabla_{\hat{v}_{i}} \hat{v}_{i} = -u^{-1} \hat{w} \hspace{0.5cm}, 
\end{equation}
and the corresponding curvature $p=A(\hat{v}_{i},\hat{v}_{i})$ is
\begin{equation} \label{rot_curv}
    p(x)= \langle w, \nu(x) \rangle u(x)^{-1}.
\end{equation}
In fact, this argument works for any unit vector $\hat{v} \in \text{span}\{ \hat{v}_{2}, \dots, \hat{v}_{n}\}$, meaning $A(\hat{v},\hat{v})=p$. This implies that $A(X,Y)=p \langle X, Y \rangle$ over $\text{span}\{ \hat{v}_{2}, \dots, \hat{v}_{n}\}$, and so the $n-1$ other principal curvatures of $N_{0}$ each equal $p$. It is convenient that these principal curvatures are encoded in \eqref{rot_curv} because it is a function on $N \setminus \{ u = 0 \}$ that depends only on $F_{0}$ and its first derivatives.  Altogether, for $x \in N$ with $u(x) > 0$ the mean curvature pulls back as

\begin{equation}
    H(x) = (n-1) p(x) + k(x).
\end{equation}

When $u(x)=0$, \eqref{rot_curv} is not well-defined, but we can continuously $p$ to this set given that $H$ is defined everywhere.
\begin{proposition} \thlabel{extension}
Let $F: N^{n} \rightarrow \mathbb{R}^{n+1}$ be a $C^{\infty}$ immersion of an oriented closed manifold $N$, and suppose $N_{0}=F_{0}(N)$ is symmetric about the axis $X_{1}$. Then the quantity $p$ defined in \eqref{rot_curv} continuously extends to the set $\{ x \in N | u(x) = 0 \}$ as $p(x)= \frac{1}{n} H(x)$.
\end{proposition}
\begin{proof}
Let $r: \mathbb{R}^{n+1} \rightarrow \mathbb{R}^{n+1}$ be a rotation about $X_{1}$. Then as a Euclidean isometry, $r$ preserves the second fundamental form of $N_{0}$, i.e.

\begin{equation} \label{pullback}
    A_{r(F_{0}(x))}(r_{*}(X),r_{*}(Y))= A_{F_{0}(x)}(X,Y)
\end{equation}
for all $x \in N$ and $X,Y \in T_{F_{0}(x)} N_{0}$. If $F_{0}(x) \in X_{1}$, let $\hat{v}_{1}, \dots, \hat{v}_{n}$ be an orthonormal basis of $T_{F_{0}(x)} N_{0}$ that diagonalizes $A$. Notice that $r(F_{0}(x))=F_{0}(x)$ for any $X_{1}$ rotation. For each $i=2, \dots, n$, we consider a rotation $r^{(i)}$ with push-forward $r^{(i)}_{*}: T_{F_{0}(x)} N_{0} \rightarrow T_{F_{0}(x)} N_{0}$ that sends $\hat{v}_{1}$ to $\hat{v}_{i}$ (such $r^{(i)}$ must exist because the group of $X_{1}$ rotations generates an $SO(n-1)$ action on $T_{F_{0}(x)} N_{0}$, and so the orbit of $\hat{v}_{1}$ is the set of all unit vectors). \eqref{pullback} then implies

\begin{equation}
    A_{F_{0}(x)}(\hat{v}_{1},\hat{v}_{1})= A_{F_{0}(x)}(\hat{v}_{i}, \hat{v}_{i}), \hspace{0.5cm} i=2, \dots, n.
\end{equation}
That is, the $n$ principal curvatures are all equal and so $\mathring{A}(x)=0$. For any sequence of points $x_{j} \in N \setminus \{ u =0 \}$ converging to this $x$ since the norm of the umbilicity tensor at $x_{j}$ approaches $0$ as $j \rightarrow \infty$, we have that $p(x_{j}) - k(x_{j}) \rightarrow 0$ and therefore $H(x_{j}) - np(x_{j}) \rightarrow 0$. The result follows.  
\end{proof}

When $p >0$ on $N_{0}$, $\langle w, \nu \rangle$ is bounded below by a multiple of the height function $u$. This allows one to realize $u$ as a function of the $x_{1}$-coordinate $\tilde{u}$, which in turn implies that $N_{0}$ is the embedded image of $\mathbb{S}^{n}$. We will not assume \textit{a priori} in our setting that $p$ is positive on $N_{0}$. However, when $N^{n}=\mathbb{S}^{n}$ and $H>0$, the \textit{converse} can also be demonstrated: if $F_{0}$ is a mean-convex embedding, then $p$ must be positive. This characterization plays an important role in understanding the long-time behavior of rotationally symmetric embedded spheres evolved by IMCF.

\begin{theorem} \thlabel{graph}
Let $F_{0}: \mathbb{S}^{n} \rightarrow \mathbb{R}^{n+1}$ be a $C^{\infty}$ rotationally symmetric immersion, and let $p$ be the quantity defined in \eqref{rot_curv} continuously extended to the set $\{ x \in \mathbb{S}^{n} | u(x)= 0 \}$ by $p(x)= \frac{1}{n}H(x)$. If $\min_{\mathbb{S}^{n}} p > 0$, then $F_{0}$ is an embedding. Conversely, if $F_{0}$ is an $H>0$ embedding, then $\min_{\mathbb{S}^{n}} p >0$.
\end{theorem}
\begin{proof}
We begin with the forward direction: if $\min_{\mathbb{S}^{n}} p > 0$, then when $u(x) \neq 0$

\begin{equation} \label{lower_bound}
    \langle w, \nu(x) \rangle \geq (\min_{\mathbb{S}^{n}} p) u(x) >0.
\end{equation}
Since $F_{0}$ is an immersion, for each $x \in \mathbb{S}^{n} \setminus \{ u =0 \}$ there is a unique vector $\tilde{v}_{1} \in T_{x} \mathbb{S}^{n}$ with
\begin{equation*}
    (F_{0})_{*}(\tilde{v}_{1})= \hat{v}_{1} \in T_{F_{0}(x)}N_{0},
\end{equation*}
for $\hat{v}_{1}$ defined in \eqref{v_1}. We compute $\frac{d}{ds} \tilde{u}(s)$ for the function $\tilde{u}$ from \eqref{x_1} along an integral curve $\gamma$ of the smooth vector field $\tilde{v}_{1}$ over $\mathbb{S}^{n} \setminus \{ u=0 \}$. The Euclidean gradient of the coordinate function $x_{1}$ equals $e_{1}$, so in view of \eqref{lower_bound},

\begin{eqnarray}
    \frac{d}{ds} \tilde{u}(s) &=& \langle e_{1}, \frac{d}{ds} (F_{0} \circ \gamma)(s) \rangle = \langle e_{1}, \hat{v}_{1} (F_{0} \circ \gamma(s)) \rangle = \langle \nu, w \rangle (s) > 0. \nonumber \label{x_1_grad}
\end{eqnarray}
Therefore, $\tilde{u}(s_{2}) > \tilde{u}(s_{1})$ whenever $s_{2} > s_{1}$. If we maximally extend $\gamma$ over an interval $(0,s_{0})$, $|\frac{d}{ds} \gamma|=|\tilde{v}_{1}|=1$ guarantees that $\gamma(s)$ converges to points $\gamma(0)$ and $\gamma(s_{0})$ in $\{ u=0 \}$ as $s$ approaches $0$ and $s_{0}$, respectively. Note that $\tilde{u}(\gamma(0)) < \tilde{u}(\gamma(s_{0}))$ as well.  

Now, suppose  $Y \in T \mathbb{R}^{n+1}$ is the Killing field corresponding to a rotation $r: \mathbb{R}^{n+1} \rightarrow \mathbb{R}^{n+1}$ that fixes $X_{1}$. For each $x \in \mathbb{S}^{n}$, let $\tilde{Y} \in T_{x} \mathbb{S}^{n}$ be the unique vector satisfying $(F_{0})_{*}(\tilde{Y})=Y \in T_{F_{0}(x)}N_{0}$. The flow of the vector field $\tilde{Y}$ yields a diffeomorphism $\tilde{r}: \mathbb{S}^{n} \rightarrow \mathbb{S}^{n}$ that satisfies $F_{0} \circ \tilde{r}= r \circ F_{0}$. 

We claim that the orbit of the above curve $\gamma: [0,s_{0}] \rightarrow \mathbb{S}^{n}$ under the set of all diffeomorphisms of this form is all of $\mathbb{S}^{n}$. Calling this orbit $R \subset \mathbb{S}^{n}$, let $f: R \rightarrow \partial B_{1}(0) \subset \mathbb{R}^{n+1}$ be a map which sends $\gamma(0)$ and $\gamma(s_{0})$ to antipodal points. Take a geodesic polar coordinate system $(\tilde{s}, \phi)$, $\phi \in \mathbb{S}^{n-1}$, with respect to $f(\gamma(0))$ of $\partial B_{1}(0) \setminus \{f(\gamma(0)), f(\gamma(s_{0}))\}$.  We can also index $\tilde{r}$ by $\tilde{r}=\tilde{r}_{\phi}$ for $\phi \in \mathbb{S}^{n-1}$, and this allows us to extend $f$ to the rest of $R$ as

\begin{equation}
    f(\tilde{r}_{\phi}(\gamma(s))) = ( \frac{\pi}{s_{0}} s, \phi ).
\end{equation}
Noting that $\tilde{r}(\gamma(0))=\gamma(0), \tilde{r}(\gamma(s_{0}))=\gamma(s_{0})$ for any $\tilde{r}$, $f$ is seen to be a continuous bijection, and so $R = \mathbb{S}^{n}$.
 
We are now ready to show $F_{0}$ is an embedding: $\{ u =0 \} = \{ \gamma(0),\gamma(s_{0})\}$ since $R=\mathbb{S}^{n}$, and $F_{0}(\gamma(0)) \neq F_{0}(\gamma(s_{0}))$ because  $\tilde{u}(\gamma(0)) \neq \tilde{u}(\gamma(s_{0}))$. For different points $x, y \in \mathbb{S}^{n} \setminus \{ u =0\}$, first suppose $x \in \text{Im}(\tilde{r}_{1} \circ \gamma)$, $y \in \text{Im}(\tilde{r}_{2} \circ \gamma)$ for $\tilde{r}_{1} \neq \tilde{r}_{2}$. One can easily verify that the curves $\tilde{r}_{1} \circ \gamma$ and $\tilde{r}_{2} \circ \gamma$ have disjoint images under $F_{0}$ for $0 < s < s_{0}$. Now suppose $x,y \in \text{Im}(\tilde{r} \circ \gamma)$ for the same $\tilde{r}$. It is also straightforward that $\tilde{r} \circ \gamma$ is an integral curve of $\tilde{v}_{1}$ given that $\gamma$ is. So in view of \eqref{x_1_grad}, $\tilde{u}(x) \neq \tilde{u}(y)$. Therefore, $F_{0}$ is injective and hence is an embedding.

For the other direction, first assume $F_{0}$ is an embedding, and let $E_{0}$ be the bounded open domain in $\mathbb{R}^{n+1}$ with $\partial E_{0} =N_{0}$.
The topological ball $E_{0}$ must intersect $X_{1}$ along an interval $(c^{-},c^{+})$. Consider $c \in (c^{-},c^{+})$, and take a point $x \in \tilde{u}^{-1}(c)$ with

\begin{equation}
    u(x)= \min_{y \in \tilde{u}^{-1}(c)} u(y).
\end{equation}
Then the $n$-dimensional disk $D$ in the hyper-plane $\{ x_{1} = c\} \subset \mathbb{R}^{n+1}$ given by

\begin{equation}
    D= \{ (c,x_{2}, \dots, x_{n+1}) | (x_{2}^{2} + \dots + x_{n+1}^{2})^{\frac{1}{2}} < u(x) \}
\end{equation}
must be contained in $E_{0}$, with $\partial D \subset N_{0}$. The outward normal of $\partial D$ as a surface in $\{ x_{1} = c \}$ equals $w$, so any vector $X \in T_{F_{0}(x)} \mathbb{R}^{n+1}$ satisfying $\langle X, w \rangle < 0$ points into $E_{0}$. Therefore, $\langle \nu, w \rangle \geq 0$ at $F_{0}(x)$.

In fact, if $H >0$ then $\langle \nu, w \rangle > 0$ at $F_{0}(x)$: suppose $\langle \nu, w \rangle(x)=0$. Then $p(x)=0$ and

\begin{eqnarray*}
    \nu(x) &=& \pm \hat{e}_{1}, \\
    \hat{v}_{1}(x) &=& \mp \hat{w},
\end{eqnarray*}
 meaning that either $\hat{v}_{1}(x)$ or $-\hat{v}_{1}(x)$ points into $D \subset E_{0}$. As a consequence, $k(x)=A(\hat{v}_{1},\hat{v}_{1}) \leq 0$, which contradicts $H>0$. So $\langle w, \nu \rangle > 0$ on $N_{0} \cap \partial D$.

To conclude the argument, we must show that there is no $y \in \tilde{u}^{-1}(\{ c \})$ with $u(y) > u(x)$, or equivalently that $N_{0} \cap \{ x_{1}=c \} = \partial D$. Call $F_{0}(x)=\gamma(s_{1})$ for some integral curve $\gamma$ of $\hat{v}_{1}$, and first suppose that there is an $s_{2} > s_{1}$ with $\tilde{u}(\gamma (s_{2})) = \tilde{u}(\gamma (s_{1}))=c$. $\langle w, \nu \rangle (s_{1}) >0$ implies $\frac{d}{ds} \tilde{u}(s_{1}) > 0$, and so there is an $s' \in (s_{1},s_{2})$ that maximizes the function $\tilde{u}$ over $(s_{1},s_{2})$. Using equation \eqref{w_grad} for $\nabla \langle w, \nu \rangle$,

\begin{eqnarray*}
    \frac{d}{ds} \tilde{u} (s') &=& \langle w, \nu \rangle (s') = 0, \\
    \frac{d^{2}}{ds^{2}} \tilde{u}(s') &=& \frac{d}{ds} \langle w, \nu \rangle (s') = - \langle e_{1}, \nu \rangle k (s') \leq 0.
\end{eqnarray*}
 The outward normal $\nu(s')$ must point toward $\{ x_{1}=c \}$; that is, $\nu(s')=-e_{1}$ and $k(s') \leq 0$. Since $p(s')=0$, this once again contradicts $H>0$. If $\tilde{u}(\gamma(s_{2})) = c$ for $s_{2} < s_{1}$, the same argument applies if we take $s' \in (s_{2},s_{1})$ to be a minimum of $\tilde{u}$. 

Altogether, $\langle w, \nu \rangle(x) >0$ whenever $\tilde{u}(x) \in (c^{-},c^{+})$. $\tilde{u}$ must then be monotone increasing along integral curves of $\hat{v}_{1}$, which in turn implies $\tilde{u}(x) \in (c^{-},c^{+})$ for every $x \in \mathbb{S}^{n} \setminus \{ u= 0 \}$. $H>0$ also guarantees the continuous extension of $p$ to $\{ u = 0 \}$ is positive.
\end{proof}

\begin{remark}
A rotationally symmetric immersion $F_{0}: \mathbb{S}^{n} \rightarrow \mathbb{R}^{n+1}$ with $\min_{\mathbb{S}^{n}} p \leq 0$ either fails to be an embedding or its mean curvature vanishes somewhere. Figure \ref{p>0} illustrates how the disk argument breaks down in each of these contexts.
\end{remark}

\begin{remark} \thlabel{degree_k}
The reverse direction also holds if $H>0$ is replaced with $f>0$ for any degree-$1$ homogeneous function $f(\lambda_{1},\lambda_{2}, \dots, \lambda_{n})$ of the principal curvatures that is non-negative when each $\lambda_{i}$ is non-negative.
\end{remark}
\begin{figure}
\begin{center}
    \textbf{Immersed Spheres with Rotational Symmetry}
    \vspace{0.3cm}
\end{center}
    \centering
    \begin{tikzpicture}[xscale=1.45,yscale=1.45,xshift=0.4cm]
    \begin{axis}[  xlabel=\footnotesize{$x_{1}$}, y=0.5cm, x=0.5cm,
          xmax=8.5, xmin=-8.5, ymax=3, ymin=-3,
          axis lines=middle,
          restrict y to domain=-7:20,
           xtick = {0},
        ytick= {0},
        yticklabels= {0, $u_{\min}(0)$},
        y axis line style={draw= none},
          enlargelimits]

        \addplot[ domain=-9:-6, smooth, thick, name path=-f] {-(4 - (x+6)*(x+6))^(0.5)};
\addplot[domain=-6:6, smooth, thick, name path=-g] {0.5*cos(1.05*deg(x+6)) + 1.5};
\addplot[ domain=6:9, smooth, thick, name path=-h] {-(4 - (x-6)*(x-6))^(0.5)}; 
\draw[thick, color=blue] (2,1.25) -- (2,-1.25);
\draw[color=blue] (2.5,0.5) node{$D$};
\draw (0,0.5) node{\large{$E_{0}$}};
\draw[->] (2,1.25) -- (2.3,1.85); 
\draw (2.3,1.85) node[anchor=west]{$\nu$};

        \addplot[ domain=-9:-6, smooth, thick, name path=f] {(4 - (x+6)*(x+6))^(0.5)};
\addplot[domain=-6:6, smooth, thick, name path=g] {-0.5*cos(1.05*deg(x+6)) -1.5};
\addplot[ domain=6:9, smooth, thick, name path=h] {(4 - (x-6)*(x-6))^(0.5)};
\addplot[gray, opacity=0.3] fill between[of =f and -f];
\addplot[gray, opacity=0.3] fill between[of =g and -g];
\addplot[gray, opacity=0.3] fill between[of =h and -h];
\draw (2,2.5) node{\small{$\langle w, \nu \rangle >0$}}; 
\draw (-7,3) node{$H>0$, embedded};
        \end{axis}
    \end{tikzpicture}

  \begin{tikzpicture}[xscale=1.7,yscale=1.25,xshift=0.4cm]
   \begin{axis}[  xlabel=\footnotesize{$x_{1}$}, y=0.5cm, x=0.5cm,
        xmax=10, xmin=-4, ymax=4, ymin=-4,
        axis lines=middle,
        restrict y to domain=-7:20,
          xtick = {0},
       ytick= {0},
       yticklabels= {0, $u_{\min}(0)$},
       y axis line style={draw= none},
         enlargelimits]
    \addplot[domain=0:1.57, name path=g]({cos(deg(2*x)) + 5.5}, {1.15*x});
     \addplot[domain=1.57:3.14, name path=h]({cos(deg(2*x)) + 5.5}, {1.15*x});
    \addplot[domain=-1.57:0, name path=-g]({cos(deg(2*x)) + 5.5}, {1.15*x});
        \addplot[domain=-3.14:-1.57, name path=-h]({cos(deg(2*x)) + 5.5}, {1.15*x});
    \addplot[domain=0:3.14, name path=m] ({1.5*cos(deg(x)) +5}, {sin(deg(x)) + 3.611});
    \addplot[domain=0:3.14, name path=k] ({0.55*sin(deg(x-1.57)) + 4.05}, {3.611 - 0.3*x});
   \addplot[domain=3.14:3.3, name path=l] ({0.55*sin(deg(x-1.57)) + 4.05}, {3.611 - 0.3*x});
     \addplot[domain=3.3:4.71, name path=j] ({0.55*sin(deg(x-1.57)) + 4.05}, {3.611 - 0.3*x});
    \addplot[domain=0: 4.05, name path=f] ({x}, {1.1*x^(0.5)});
    \addplot[domain=0: 4.05, name path=-f] {-(1.1*x^(0.5))};
     \addplot[domain=0:3.14, name path=-m] ({1.5*cos(deg(x)) +5}, {-sin(deg(x)) - 3.611});
    \addplot[domain=0:3.14, name path=-k] ({0.55*sin(deg(x-1.57)) + 4.05}, {-3.611 + 0.3*x});
      \addplot[domain=3.14:3.3, name path=-l] ({0.55*sin(deg(x-1.57)) + 4.05}, {-3.611 + 0.3*x});
        \addplot[domain=3.3:4.71, name path=-j] ({0.55*sin(deg(x-1.57)) + 4.05}, {-3.611 + 0.3*x});
    \addplot[gray, opacity=0.3] fill between[of =f and -f];
    \addplot[gray, opacity=0.3] fill between[of =g and -g];
    \addplot[gray, opacity=0.3] fill between[of =j and -j];
     \addplot[gray, opacity=0.3] fill between[of =h and l];
    \addplot[gray, opacity=0.3] fill between[of =k and m];
    \addplot[gray, opacity=0.3] fill between[of =-h and -l];
    \addplot[gray, opacity=0.3] fill between[of =-k and -m];
    \draw (-3,4) node{$H \ngtr 0$, embedded};
    \draw (3,0.75) node{\large{$E_{0}$}};
    \draw[color=blue] (4.5,0.5) node[anchor=west]{$D$};
    \draw[thick, color=blue] (4.5,1.9) -- (4.5,-1.9);
    \draw[->] (4.5,1.9) -- (5.25,1.9);
    \draw (5.5,1.9) node[anchor=south]{$\nu$};
    \draw (5.5,1.5) node[anchor=west]{$\langle w,\nu \rangle=0$, $k=0$};
    \end{axis}
   \end{tikzpicture}

    \begin{tikzpicture}[xscale=1.45,yscale=1.45,xshift=0.4cm]
    \begin{axis}[  xlabel=\footnotesize{$x_{1}$}, y=0.5cm, x=0.5cm,
          xmax=8.5, xmin=-8.5, ymax=3, ymin=-3,
          axis lines=middle,
          restrict y to domain=-7:20,
           xtick = {0},
        ytick= {0},
        yticklabels= {0, $u_{\min}(0)$},
        y axis line style={draw= none},
          enlargelimits]
          \addplot [domain=-1.5:1.5]({0.5*x^3-0.5*x},{0.5*x^2+1.5});
                  \addplot[ domain=-7.242:-0.9375, smooth, name path=-f] {(9 - 0.5*(x+3)*(x+3))^(0.5)};
                  \addplot[ domain=0.9375:7.242, smooth, name path=f] {(9 - 0.5*(x-3)*(x-3))^(0.5)};
                  \addplot [domain=-1.5:1.5]({0.5*x^3-0.5*x},{-0.5*x^2-1.5});
                  \addplot[ domain=-7.242:-0.9375, smooth, name path=-f] {-(9 - 0.5*(x+3)*(x+3))^(0.5)};
                  \addplot[ domain=0.9375:7.242, smooth, name path=f] {-(9 - 0.5*(x-3)*(x-3))^(0.5)};
                  \draw[thick, color=blue] (0,-1.5) -- (0,1.5);
                  \draw (-7,3.25) node{\scriptsize{$H>0$, not embedded}};
                  \draw[->] (0,1.5) -- (0,0.75);
                  \draw (0,1) node[anchor=east]{\scriptsize{$\nu$}};
                  \draw (0,1) node[anchor=west]{\scriptsize{$\langle w, \nu \rangle < 0$, $k>0$}};
                  \draw[color=blue] (0,-0.5) node[anchor=west]{\scriptsize{$D$}};
    \end{axis}
\end{tikzpicture}
    \caption{The cross sections of rotationally symmetric immersions. If the surface is mean-convex and embedded then $p >0$, but both of these assumptions are neccessary.}
    \label{p>0}
\end{figure}

\section{Short-Time Existence and Embeddedness}
We now consider the evolution problem \eqref{IMCF} for a smooth, strictly mean-convex embedding $F_{0}: \mathbb{S}^{n} \rightarrow \mathbb{R}^{n+1}$ with rotationally symmetric image. In \cite{Gerhard1999}, Huisken and Polden establish the existence of a short-time solution to \eqref{IMCF} for smooth initial data on a closed manifold. 

\begin{theorem}[Short-Time Existence, Theorem 7.17 in \cite{Gerhard1999}] \thlabel{short_time}
Let $F_{0}: N^{n} \rightarrow \mathbb{R}^{n+1}$ be a $C^{\infty}$ immersion of a smooth, oriented closed manifold $N$ with positive outward mean curvature $H$. Then there exists a unique one-parameter family of $C^{\infty}$ immersions $F: N^{n} \times [0,T) \rightarrow \mathbb{R}^{n+1}$ defined over a time interval $0 \leq t < T$ that obeys \eqref{IMCF} and satisfies $F(x,0)=F_{0}(x)$.
\end{theorem}

One consequence of the uniqueness of solutions is that if an isometry $f$ of the ambient space fixes $N_{0}=F_{0}(N)$, i.e. if $f(N_{0})=N_{0}$, then $f(N_{t})=N_{t}$ for each $t \in [0,T_{\max})$. Therefore, each $N_{t}$ is rotationally symmetric about $X_{1}$ given that $N_{0}$ is, so we can consider the evolution of the quantities introduced in the previous section under IMCF.

A key way in which inverse mean curvature flow differs from mean curvature flow is that it does not in general preserve embeddedness-- that is, $F_{t}$ may not be an embedding even if $F_{0}$ is. One example of this phenomenon is when $N_{0}$ is chosen to be the boundary of two disjoint unit balls in Euclidean space. The expanding round spheres will exist for a long time, but eventually intersect one another. \cite{me} also provides an example of a self-intersection when $F_{0}$ is an embedding of $\mathbb{S}^{n}$. In our context, however, \thref{graph} allows us to rule out this possibility.

\begin{theorem}[Preserving Embeddedness] \thlabel{embed}
Let $F_{0}: \mathbb{S}^{n} \rightarrow \mathbb{R}^{n+1}$ be an $H>0$, $C^{\infty}$ rotationally symmetric embedding, and $F: \mathbb{S}^{n} \times [0,T_{\max}) \rightarrow \mathbb{R}^{n+1}$ the corresponding maximal solution to \eqref{IMCF}. Then for $t \in [0,T_{\max})$, $F_{t}$ is an embedding. 
\end{theorem}

\begin{proof}
Utilizing the first part of \thref{graph}, we will show that $\min_{x \in \mathbb{S}^{n}} p(x,t) >0$ for each $t \in [0,T_{\max})$. Since $F_{0}$ is an $H>0$ embedding, $\min_{x \in \mathbb{S}^{n}} p(x,0) > 0$ by the second part of \thref{graph} and so $\min_{x \in \mathbb{S}^{n}} p(x,t) >0$ for $t \in [0,\epsilon)$ by continuity. Let 
\begin{equation*}
    t_{0}= \sup \{ t \in [0,T_{\max}) | \min_{x \in \mathbb{S}^{n}} p(x,t) > 0 \},
\end{equation*}
and suppose $t_{0} < T_{\max}$. Since $F_{t_{0}}$ is an immersion with positive mean curvature, $p (x,t_{0}) > 0$ whenever $u(x,t_{0})=0$ by \thref{extension}. So there is some $x_{0} \in \mathbb{S}^{n}$ with $u(x_{0},t_{0}) > 0$ and $p(x_{0},t_{0})=0$. 

$\min_{x \in \mathbb{S}^{n}} p(x,t_{0})=0$ and so $\nabla p (x_{0},t_{0})=0$. Let us compute $\nabla p(x,t)$ for an arbitrary $(x,t) \in (\mathbb{S}^{n} \times [0,T)) \setminus \{ u =0 \}$ (we will need this formula later anyway). Take an orthonormal basis $\hat{v}_{1}, \dots, \hat{v}_{n}$ at $(x,t)$ with $\hat{v}_{1}$ from \eqref{v_1}. By the identity

\begin{equation} \label{sum}
    |\langle w, \nu \rangle|^{2} = 1 - |\langle e_{1}, \nu \rangle|^{2} 
\end{equation}
we find in this basis

\begin{eqnarray} \label{w_grad}
    \nabla \langle w, \nu \rangle &=& \nabla_{i} \langle w, \nu \rangle \hat{v}_{i}= - \frac{\langle e_{1}, \nu \rangle}{\langle w, \nu \rangle} \nabla_{i} \langle e_{1}, \nu \rangle \hat{v}_{i} \\
    &=&  -\frac{\langle e_{1}, \nu \rangle}{\langle w, \nu \rangle} \langle e_{1}, \nabla_{i} \nu \rangle \hat{v}_{i} 
    = - \langle e_{1}, \nu \rangle k \hat{v}_{1}. \nonumber
\end{eqnarray}
 On the other hand, the Euclidean gradient of the function $(x_{2}^{2} + \dots + x_{n+1}^{2})^{\frac{1}{2}}$ over $\mathbb{R}^{n+1} \setminus X_{1}$ equals $\hat{w}$, and so projecting this onto the tangent space of $N_{t}$ yields

\begin{eqnarray} \label{height_grad}
    \nabla u &=& w - \langle w, \nu \rangle \nu \nonumber \\
             &=&(1 - |\langle w, \nu \rangle|^{2}) w - \langle w, \nu \rangle \langle e_{1}, \nu \rangle e_{1} \\
             &=& \langle e_{1}, \nu \rangle ( \langle e_{1}, \nu \rangle w - \langle w, \nu \rangle e_{1})= -\langle e_{1}, \nu \rangle \hat{v}_{1}, \nonumber
\end{eqnarray}
where we once again used \eqref{sum}. Putting these together, we get

\begin{equation} \label{p_grad}
   \nabla p = \langle e_{1}, \nu \rangle u^{-1} (p-k) \hat{v}_{1}.
\end{equation}
Therefore, $\nabla p (x_{0},t_{0})=0$ implies that $p(x_{0},t_{0})=k(x_{0},t_{0})=H(x_{0},t_{0})=0$, which is a contradiction. Conclude that $\min_{x \in \mathbb{S}^{n}} p(x,t) >0$ and $F_{t}$ is an embedding as long as the solution exists by the forward statement in \thref{graph}.
\end{proof}

\begin{remark}
Recalling \thref{degree_k}, this argument applies more generally to rotationally symmetric flows of the form $\frac{\partial F}{\partial t}(x,t) = f^{-1}(\lambda_{1}, \dots, \lambda_{n}) \nu(x,t)$, $f>0$, for a degree-$1$ homogeneous function $f$ that is non-negative when each $\lambda_{i} \geq 0$. In particular, the expanding flows considered in \cite{Urbas} and \cite{Guan2009TheQI} stay embedded in the rotationally symmetric context.
\end{remark}

\begin{remark}
Given that $\min_{x \in \mathbb{S}^{n}} p(x,t) > 0$, each $N_{t}$ can also be identified with a graph in the upper half-plane which generates it by revolution, and the questions of long-time existence and convergence may be approached by studying the evolution of these graphs. This approach involves a different gauge choice from \eqref{IMCF}, though, and we found the standard gauge to be more natural in this case. 
\end{remark}

It is shown in \cite{me} that an embedded solution of \eqref{IMCF} becomes star-shaped by the time 
\begin{equation}
t_{*}=n\text{log}(R^{-1}\text{diam} (N_{0})), 
\end{equation}
where $R$ is the radius of the largest ball enclosed by $N_{0}$ and $\text{diam}(N_{0})$ its extrinsic diameter. Therefore, whenever the solution to \eqref{IMCF} for a smooth rotationally symmetric embedding $F_{0}$ of $\mathbb{S}^{n}$ exists for a time $T_{\max} > t_{*}$, we have that $T_{\max}=+\infty$ and $N_{t}$ is star-shaped for $t>t_{*}$. In turn, \cite{claus}, \cite{huisken2}, and \cite{Urbas} provide stronger estimates and guarantee the asymptotic roundness of $N_{t}$. With the asymptotic behavior of a long-time embedded solution of IMCF already understood, the rest of this paper focuses on proving long-time existence.

\section{Non-Cylindrical Spacetime Domains}
The key to the regularity theory for IMCF is a lower bound on the mean curvature $H$: estimating $H$ from below (or, equivalently, the flow speed from above) uniformly over any given finite time interval will guarantee long-time existence, see Theorem 2.2 in \cite{huisken2}. Maximum principles are an obvious approach here, but a difficult issue is that many quantities one would like to exploit, e.g. the height function or the $\hat{w}$-component of the normal vector, are either $0$ or are undefined on $N_{0} \cap X_{1}$. This suggests that one should consider regions of $N_{t}$ close to $X_{1}$ and away from $X_{1}$ separately. Indeed, the literature for flows of these types of surfaces contains more than one approach to this issue, see \cite{angenent}, \cite{maria2}, \cite{maria3}, \cite{head2019singularity}, and \cite{lumer}. 

Our approach is inspired by \cite{maria2}, and it involves defining separate regions of $N_{t}$ and pulling them back via the embedding $F_{t}$. We can distinguish between points close to and away from $X_{1}$ in the following way: from the proof of \thref{graph}, we know $N_{0}$ intersects $X_{1}$ at two points, with $\nu=+\hat{e}_{1}$ at one of these points and $-\hat{e}_{1}$ at the other. We consider the subsets of $N_{0}$ where the $\hat{e}_{1}$-component of $\nu$ is non-negative and non-positive, respectively, and take the connected components of each that contain either of these points.
\begin{definition} \thlabel{subset}
Given a $C^{\infty}$, $H>0$ rotationally symmetric embedding $F_{0}: \mathbb{S}^{n} \rightarrow \mathbb{R}^{n+1}$, the \textbf{right cap} $C_{0}^{+}$ of $F_{0}$ is the interior of the connected component of the set $\{ x \in \mathbb{S}^{n}| \langle \nu(x), e_{1} \rangle \geq 0 \}$ that intersects $u^{-1}(\{0\})$. The \textbf{left cap} $C_{0}^{-}$ of $F_{0}$ is the interior of the connected component of  $\{ x \in \mathbb{S}^{n} | \langle \nu(x), e_{1} \rangle \leq 0 \}$ that intersects $u^{-1}(\{ 0 \})$. The \textbf{bridge} $B_{0}$ of $F_{0}$ is the interior of $\mathbb{S}^{n} \setminus (C^{+}_{0} \cup C^{-}_{0})$.
\end{definition}
\vspace{2cm}

\begin{figure}[h]
\begin{center}
    \textbf{Different Regions of $N_{0}$}
    \vspace{0.3cm}
\end{center}
    \centering
    \begin{tikzpicture}[xscale=1.45,yscale=1.45,xshift=0.4cm]
    \begin{axis}[  xlabel=\footnotesize{$x_{1}$}, y=0.5cm, x=0.5cm,
          xmax=8.5, xmin=-8.5, ymax=3, ymin=-3,
          axis lines=middle,
          restrict y to domain=-7:20,
           xtick = {0},
        ytick= {0},
        yticklabels= {0, $u_{\min}(0)$},
        y axis line style={draw= none},
          enlargelimits]

        \addplot[ domain=-9:-6, smooth, thick, name path=-f] {-(4 - (x+6)*(x+6))^(0.5)};
\addplot[domain=-4.5:-1.5, smooth, thick, name path=-g] {0.5* cos(deg(1.047*x + 1.047*6)) + 1.5};
\addplot[ domain=-6:-4.5, smooth, thick, name path=-h] {0.5 *cos(deg(1.047*x + 1.047*6)) + 1.5};
\addplot[ domain=-1.5:0, smooth, thick, name path=-i] {0.5 *cos(deg(1.047*x + 1.047*6)) + 1.5};
\addplot[domain=0:1.5, smooth, thick, name path=-j] {0.7*cos(deg(1.047*x + 1.047*6)) +1.3};
\addplot[domain=1.5:4.5, smooth, thick, name path=-k] {0.7*cos(deg(1.047*x + 1.047*6)) +1.3};
\addplot[domain=4.5:6, smooth, thick, name path=-l] {0.7*cos(deg(1.047*x + 1.047*6)) +1.3};
\addplot[ domain=6:9, smooth, thick, name path=-m] {(4 - (x-6)*(x-6))^(0.5)}; 
\draw (2,2) node{$B_{0}$};
\draw[color=blue] (7,2.5) node{$C^{+}_{0}$};
\draw[color=blue] (-7,2.5) node{$C^{-}_{0}$};

        \addplot[ domain=-9:-6, smooth, thick, name path=f] {(4 - (x+6)*(x+6))^(0.5)};
\addplot[domain=-4.5:-1.5, smooth, thick, name path=g] {-0.5* cos(deg(1.047*x + 1.047*6)) - 1.5};
\addplot[ domain=-6:-4.5, smooth, thick, name path=h] {-0.5 *cos(deg(1.047*x + 1.047*6)) - 1.5};
\addplot[ domain=-1.5:0, smooth, thick, name path=i] {-0.5 *cos(deg(1.047*x + 1.047*6)) - 1.5};
\addplot[domain=0:1.5, smooth, thick, name path=j] {-0.7*cos(deg(1.047*x + 1.047*6)) -1.3};
\addplot[domain=1.5:4.5, smooth, thick, name path=k] {-0.7*cos(deg(1.047*x + 1.047*6)) -1.3};
\addplot[domain=4.5:6, smooth, thick, name path=l] {-0.7*cos(deg(1.047*x + 1.047*6)) -1.3};
\addplot[ domain=6:9, smooth, thick, name path=m] {-(4 - (x-6)*(x-6))^(0.5)}; 

\addplot[domain=6:6.2, smooth, dashed]{2*(1- 25*(x-6)^(2))^(0.5)};
\addplot[domain=5.8:6, smooth, dashed, name path=n]{2*(1- 25*(x-6)^(2))^(0.5)};
\addplot[domain=6:6.2, smooth, dashed]{-2*(1- 25*(x-6)^(2))^(0.5)};
\addplot[domain=5.8:6, smooth, dashed, name path=-n]{-2*(1- 25*(x-6)^(2))^(0.5)};
\addplot[domain=-6:-5.8, smooth, dashed, name path=o]{2*(1- 25*(x+6)^(2))^(0.5)};
\addplot[domain=-6.2:-6, smooth, dashed]{2*(1- 25*(x+6)^(2))^(0.5)};
\addplot[domain=-6:-5.8, smooth, dashed, name path=-o]{-2*(1- 25*(x+6)^(2))^(0.5)};
\addplot[domain=-6.2:-6, smooth, dashed]{-2*(1- 25*(x+6)^(2))^(0.5)};
\addplot[blue, opacity=0.3] fill between[of =f and -f];
\addplot[blue, opacity=0.3] fill between[of =m and -m];
\addplot[blue, opacity=0.3] fill between[of =n and -n];
\addplot[blue, opacity=0.3] fill between[of =o and -o];

\addplot[domain=-3.15:-2.85, smooth, dashed]{(1-44.44*(x+3)^(2))^(0.5)};
\addplot[domain=-3.15:-2.85, smooth, dashed]{-(1-44.44*(x+3)^(2))^(0.5)};
\addplot[domain=2.9:3.1, smooth, dashed]{0.6*(1-100*(x-3)^(2))^(0.5)};
\addplot[domain=2.9:3.1, smooth, dashed]{-0.6*(1-100*(x-3)^(2))^(0.5)};
\addplot[domain=-0.2:0.2, smooth, dashed]{2*(1-25*(x)^(2))^(0.5)};
\addplot[domain=-0.2:0.2, smooth, dashed]{-2*(1-25*(x)^(2))^(0.5)};
        \end{axis}
    \end{tikzpicture}
    \caption{The caps $C^{\pm}_{0}$ and bridge $B_{0}$ of a rotationally symmetric $N_{0}$.}
    \label{regions}
\end{figure}
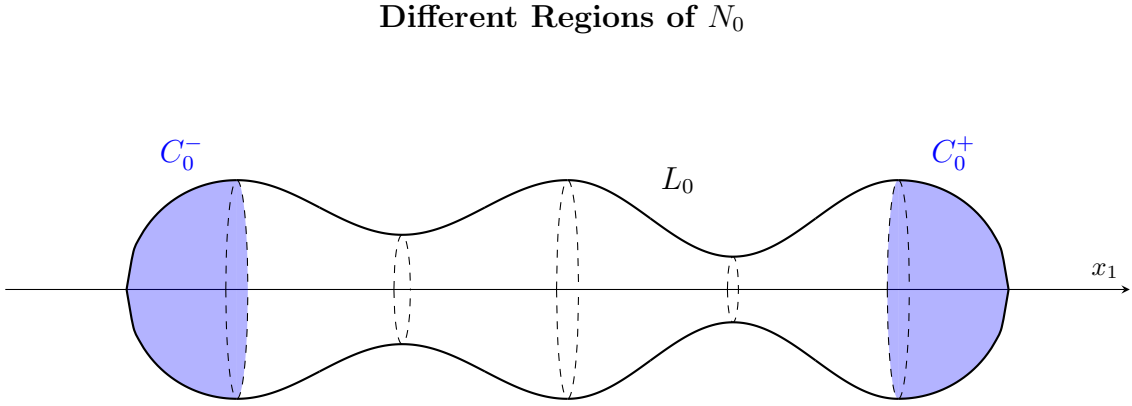

For $B_{0} = \varnothing$-- the case of no ``necks" -- there is a choice of origin on $X_{1} \subset \mathbb{R}^{n+1}$ so that, considering the position vector $\vec{F}_{0}$ with repsect to this point, we have $\langle \nu (x), e_{1} \rangle \leq 0$ when $\langle \vec{F_{0}}(x), e_{1} \rangle = \tilde{u}(x) \leq 0$ and $\langle \nu (x), e_{1} \rangle \geq 0$ when $\tilde{u} \geq 0$. The support function must then satisfy

\begin{equation} \label{supp}
    \langle \vec{F_{0}}(x), \nu \rangle = \tilde{u}(x) \langle \nu (x), e_{1} \rangle + u(x) \langle w, \nu (x) \rangle > 0,
\end{equation}
and so $N_{0}$ is star-shaped-- this ensures global existence for the corresponding $N_{t}$. Thus the problem is non-trivial only when $B_{0} \neq \varnothing$. 

When $B_{0} \neq \varnothing$, we impose an assumption on $N_{0}$ in addition to rotational symmetry. We require that the ratio of the highest and lowest values of $p$ over $\overline{B}_{0}$ is no larger than $n^{\frac{n}{2(n-1)}}$, which will be crucial for controlling the flow speed for the IMCF of $N_{0}$.

\begin{definition} \thlabel{admissible}
An $C^{\infty}$, $H>0$ rotationally symmetric embedding $F_{0}: \mathbb{S}^{n} \rightarrow \mathbb{R}^{n+1}$ is \textbf{admissible} if the principal curvature $p$ of rotation satisfies

\begin{equation} \label{neck}
    \frac{\max_{\overline{B}_{0}} p}{\min_{\overline{B}_{0}} p} < n^{\frac{n}{2(n-1)}}.
\end{equation}
\end{definition}

Although many of the results throughout this paper apply more generally to the IMCF of any $H>0$ rotationally symmetric embedded sphere, this additional $C^{1}$ assumption is needed to obtain control on the relevant geometric quantities over the bridge $B_{t}$ of $N_{t}$. Condition \eqref{neck} places an upper bound on the maximum over $B_{0}$ of the quantity
\begin{equation} \label{v}
    v= \langle w, \nu \rangle^{-1}.
\end{equation}
An upper bound on $v$ limits how ``narrow" the necks of $N_{0}$ are, in the sense that $v$ measures how quickly the height function is changing in the $\hat{v}_{1}$ direction. The condition also requires that the necks cannot be too ``thin", specifically that the ratio of the largest and smallest values of the height $u$ over $B_{0}$ does not exceed $n^{\frac{n}{2(n-1)}}$. Conversely, whenever this ratio for $u$ is smaller than $n^{\frac{n}{2(n-1)}}$ on a rotationally symmetric $N_{0}$, dilating $N_{0}$ in the $x_{1}$ direction by a sufficiently large factor produces an admissible rotationally symmetric surface. Since the IMCF of star-shaped $N_{0}$ is already understood, we demonstrate in Appendix A.2 the existence of an admissible $N_{0}$ which is not star-shaped.

Let us now explain the approach to a localized version of the maximum principle. The solution to \eqref{IMCF} for a rotationally symmetric sphere $N_{0}$ is a one-parameter family of embeddings $F: \mathbb{S}^{n} \times [0,T) \rightarrow \mathbb{R}^{n+1}$. Define the open subsets $C^{+} \subset  \mathbb{S}^{n} \times [0,T)$, $C^{-} \subset \mathbb{S}^{n} \times [0,T)$, and $B \subset \mathbb{S}^{n} \times [0,T)$ by

\begin{eqnarray}
    C^{+}&=& \text{Int}(\cup_{t \in (0,T)} C^{+}_{t} \times \{ t \}), \label{domain1} \\
    C^{-}&=& \text{Int}(\cup_{t \in (0,T)} C^{-}_{t} \times \{t\}), \label{domain2} \\
    B &=& \text{Int}(\cup_{t \in (0,T)} B_{t} \times \{t\}). \label{domain3}
\end{eqnarray}

Here $C^{\pm}_{t}$ and $B_{t}$ are the caps and bridges of \thref{subset} for each embedding $F_{t}$. We will apply maximum principles to each of these domains separately, which introduces a boundary to the problem. These boundaries are especially complicated because they may be non-cylindrical.

\begin{definition}
For a closed manifold $N$ and an open domain $U \subset N \times (0,T)$, let $U_{t}= U \cap (N \times \{t\})$ for $t \in (0,T)$, $U_{0} = \overline{U} \cap (N \times \{ 0\})$, and $U_{T}=\overline{U} \cap (N \times \{T\})$. The \textbf{parabolic boundary} $\partial_{P} U$ of $U$ is  $\partial_{P} U = \partial U \setminus U_{T}$, where $\partial U$ is the topological boundary of $U$ in $N \times [0,T)$. \\ \\ The \textbf{reduced parabolic boundary} $\tilde{\partial}_{P} U$ of $U$ is $\tilde{\partial_{P}} U= U_{0} \cup (\cup_{0 \leq t < T} \partial U_{t})$, where $\partial U_{t}$ is the topological boundary of $U_{t}$ in $N \times \{t\}$.
\end{definition}

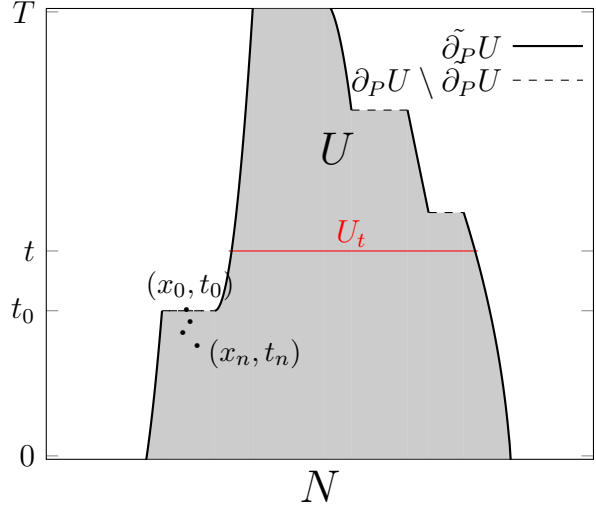
\begin{figure}
\begin{center}
    \textbf{The Reduced Parabolic Boundary of $U$}
    \vspace{0.3cm}
\end{center}
\centering
   \begin{tikzpicture}[ xscale=1.05,yscale=1.05,xshift=-1cm]
\begin{axis}     [ 
      xlabel={\Large{$N$}},
          xmax=5, xmin=-1.5, ymax=1.4, ymin=0.3,
          restrict y to domain=-2:1.8,
           ytick = {0.2, 0.625, 0.8, 1.5},
        yticklabels = {$0$, $t_{0}$, $t$,$T$},
        xtick= {-3},
          enlargelimits]

\draw[red] (0.45,0.8) -- (4,0.8);
\draw(2.2,0.85) node{\color{red}{$U_{t}$}};
\draw(2,1.1) node{\Large{$U$}};

\draw(-0.15,0.625) node{\Large{$\cdot$}};

\draw[thick] (4.5,1.4) -- (5.5,1.4);
\draw[dashed] (4.5,1.3) -- (5.5,1.3);

\draw(4.5,1.4) node[anchor=east]{$\tilde{\partial_{P}} U$};
\draw(4.5,1.3) node[anchor=east]{$\partial_{P} U \setminus \tilde{\partial_{P}} U$};

\addplot [domain= -0.4:0.1, smooth, dashed, name path=z] {0.625};
\draw (0,0.58) node[anchor=north]{\Large{$\cdot$}};
\draw (-0.2,0.62) node[anchor=north]{\Large{$\cdot$}};
\draw (-0.1,0.65) node[anchor=north]{\Large{$\cdot$}};
\draw (0,0.5) node[anchor=west]{\small{$(x_{n},t_{n})$}};
\draw (-0.1,0.625) node[anchor=south]{\small{$(x_{0},t_{0})$}};
\addplot [domain=-1:-0.5, smooth, thick, name path=f] {2.5*(x+1)^(2)};
\addplot [domain=-0.5:0.25, smooth, dashed, name path=g] {0.625};
\addplot [domain=0.25:0.8, smooth, thick, name path=h] {3*(x-0.25)^(2)+0.625};
\addplot [domain=0.8:1.8, smooth, thick, name path=i] {1.5325};
\addplot [domain=1.8:2.2, smooth, thick, name path=j] {-2*(x-1.8)^(2)+1.5325};
\addplot [domain=2.2:3, smooth, dashed, name path=k] {1.2125};
\addplot[domain=3:3.3, smooth, thick, name path=l] {-(x-3)+1.2125};
\addplot[domain=3.3:3.8, smooth, dashed, name path=m] {0.9125};
\addplot[domain=3.8:4.5, smooth, thick, name path=n] {1.091*(4.5-x)^(0.5)};

\addplot[domain=-1:-0.5, smooth, thick, name path=-f] {0};
\addplot[domain=-0.5:0.25, smooth, thick, name path=-g] {0};
\addplot[domain=0.25:0.8, smooth, thick, name path=-h] {0};
\addplot [domain=0.8:1.8, smooth, thick, name path=-i] {0};
\addplot [domain=1.8:2.2, smooth, thick, name path=-j] {0};
\addplot [domain=2.2:3, smooth, thick, name path=-k] {0};
\addplot[domain=3:3.3, smooth, thick, name path=-l] {0};
\addplot[domain=3.3:3.8, smooth, thick, name path=-m] {0};
\addplot[domain=3.8:4.5, smooth, thick, name path=-n] {0};

\addplot [
      thick,
     fill=gray!,opacity=0.4
    ] fill between[of=-f and f];

\addplot [
      thick,
     fill=gray!,opacity=0.4
    ] fill between[of=-g and g];
    
\addplot [
      thick,
     fill=gray!,opacity=0.4
    ] fill between[of=-h and h];
    
\addplot [
      thick,
     fill=gray!,opacity=0.4
    ] fill between[of=-i and i];

\addplot [
      thick,
     fill=gray!,opacity=0.4
    ] fill between[of=-j and j];
    
\addplot [
      thick,
     fill=gray!,opacity=0.4
    ] fill between[of=-k and k];
    
\addplot [
      thick,
     fill=gray!,opacity=0.4
    ] fill between[of=-l and l];

\addplot [
      thick,
     fill=gray!,opacity=0.4
    ] fill between[of=-m and m];

\addplot [
      thick,
     fill=gray!,opacity=0.4
    ] fill between[of=-n and n];
\end{axis}
    \end{tikzpicture}
    
    \caption{For a non-cylindrical domain $U \subset N \times [0,T)$, the reduced parabolic boundary $\tilde{\partial}_{P} U$ does not include the dotted parts of $\partial_{P} U$. If every point in $\partial_{P} U \setminus \tilde{\partial_{P}}U$ can be approached in $U$ from below in time, $\sup_{U} f$ cannot occur at any of these points.}
    \label{non_cylinder}
\end{figure}

In general, $\tilde{\partial}_{P} U \neq \partial_{P} U$, see Figure \ref{non_cylinder}. In this setting, the domains \eqref{domain1}-\eqref{domain3} are non-cylindrical when the $x_{1}$-coordinate of the left-most or right-most maxima in $\mathbb{R}^{n+1}$ of the height $u$ are not continuous functions of time. This may happen as maxima and minima of $u$ may spontaneously form or disappear during the evolution. With this possibility in mind, we employ a modified version of the maximum principle over such a domain detailed in \cite{maria3}, see also \cite{lumer} and \cite{head2019singularity}.

\begin{theorem}[Non-Cylindrical Maximum Principle] \thlabel{noncyl}
Let $F: N^{n} \times [0,T) \rightarrow \mathbb{R}^{n+1}$ be a solution of the Inverse Mean Curvature Flow \eqref{IMCF} over a closed manifold $N$. For a domain $U \subset N \times [0,T)$ and $f \in C^{2,1}(U) \cap C(\overline{U})$, suppose for a smooth vector field $\eta$ over $U$ we have

\begin{equation*}
    (\partial_{t} - \frac{1}{H^{2}} \Delta) f \leq \langle \eta, \nabla f \rangle.
\end{equation*}
(Resp. $\geq$ at a minimum) Here $\Delta$ and $\nabla$ are the Laplacian and gradient operators over $N_{t}$, respectively. Then

\begin{equation*}
    \sup_{U} f \leq \sup_{\partial_{P} U} f
\end{equation*}
(Resp. $\inf_{U} f \geq \inf_{\partial_{P} U} f$). Furthermore, suppose that $f$ has a positive supremum over $U$ and that each $(x_{0},t_{0}) \in \partial_{P} U \setminus \tilde{\partial}_P U$ is a limit point of $U \cap \{ t < t_{0} \}$. Then
\begin{equation*}
    \sup_{U} f \leq \sup_{\tilde{\partial}_{P} U} f
\end{equation*}
(Resp. $\inf_{U} f \geq \inf_{\tilde{\partial}_{P} U} f$ for a positive minimum).
\end{theorem}

The condition that a sequence from $U$ approaches each point in $\partial_{P} U \setminus \tilde{\partial}_{P} U$ from below in time is not included in \cite{head2019singularity} and \cite{maria3}, but it seems to be necessary for the second part of the statement due to the time asymmetry of parabolic equations. For this reason, we prove \thref{noncyl} with the domain geometry assumption in mind in Appendix A.1. 

We claim that the domain $B$ defined in \eqref{domain3} satisfies this additional requirement on domain geometry. $C^{+}_{t}$ and $C^{-}_{t}$ are connected open sets for each $t$, and so the image of the function $\tilde{u}$ over each of these is a connected interval. Essentially, we must show that these these intervals cannot instantaneously shrink as time progresses. 

\begin{proposition}
Let $F_{0}: \mathbb{S}^{n} \rightarrow \mathbb{R}^{n+1}$ be an $H>0$, rotationally symmetric embedding, and $F: \mathbb{S}^{n} \times [0,T_{\max}) \rightarrow \mathbb{R}^{n+1}$ the corresponding maximal solution to \eqref{IMCF}. For each $t \in [0,T_{\max})$, define $a(t), b(t) \in \mathbb{R}$ by $C^{-}_{t}= \{ x \in \mathbb{S}^{n} | \tilde{u}(x,t) < a(t) \}$ and $C^{+}_{t} =\{x \in \mathbb{S}^{n} | \tilde{u}(x,t) > b(t) \}$, respectively. Then for each $t_{0} \in [0,T_{\max})$,

\begin{equation} \label{liminf}
     a(t_{0}) \leq \liminf_{t \searrow t_{0}} a(t),
\end{equation}
resp. $b(t_{0}) \geq \limsup_{t \searrow t_{0}} b(t)$.

As a consequence, for any $x_{0} \in C^{-}_{t_{0}}$ (resp. $C^{+}_{t_{0}}$), there is a neighborhood $\mathcal{O}$ of $(x_{0},t_{0})$ in $\mathbb{S}^{n} \times [0,T_{\max})$ such that $\mathcal{O} \cap \{ t > t_{0} \} \subset C^{-}$ (resp. $C^{+}$). 
\end{proposition}
\begin{proof}
 We present the proof for $a(t)$, as the proof for $b(t)$ is identical. Suppose $\liminf_{t \searrow t_{0}} a(t) < a(t_{0})$. $\langle e_{1}, \nu \rangle (x,t_{0}) \leq 0$ whenever $\liminf_{t \searrow t_{0}} a(t) < \tilde{u}(x,t_{0}) < a(t_{0})$, but $\langle e_{1}, \nu \rangle \neq 0$ at all of these points because the height function $u(x,t_{0})$ cannot be constant over a positive-measure subset of $\mathbb{S}^{n}$ (Lemma 4.7 in \cite{angenent} establishes this in the graphical gauge of a rotationally symmetric solution of mean curvature flow, and applying the Sturmian Theorem for inverse mean curvature flow yields the same conclusion). Then there must exist a $c \in \mathbb{R}$ so that $\liminf_{t \searrow t_{0}} a(t) < c < a(t_{0})$ and $\langle e_{1}, \nu \rangle (x,t_{0}) < 0$ when $\tilde{u}(x,t_{0})=c$. Define the spatial domain $U \subset C^{-}_{t_{0}} \subset \mathbb{S}^{n}$ by
 
 \begin{equation}
     U = \{ x \in \mathbb{S}^{n} | \tilde{u}(x,t_{0}) < c \}.
 \end{equation}
 Then
 
 \begin{eqnarray}
      \langle e_{1}, \nu \rangle (x,t_{0}) &\leq& 0 \hspace{0.5cm} x \in U, \\
      \langle e_{1}, \nu \rangle (x,t_{0}) &<& 0 \hspace{0.5cm} x \in \partial U, \\
      \{ x \in \mathbb{S}^{n} | \tilde{u}(x,t_{0}) &=& \liminf_{t \searrow t_{0}} a(t) \} \subset U. \label{U}
 \end{eqnarray}
By continuity of $\langle e_{1}, \nu \rangle$ in time, we may choose a time interval $[t_{0},t_{0} + \epsilon)$ so that $\langle e_{1}, \nu \rangle (x,t) < 0$ for $(x,t) \in \partial U \times [t_{0}, t_{0} + \epsilon)$. We will show that $\langle e_{1}, \nu \rangle (x,t) \leq 0$ on $U \times [t_{0},t_{0} + \epsilon)$. First of all, the evolution equation \eqref{evolution_e} for $\langle e_{1}, \nu \rangle$ from \thref{evolution_equations} reads
\begin{equation*} 
    (\partial_{t} - \frac{1}{H^{2}} \Delta) \langle \nu, e_{1} \rangle (x,t)= \frac{|A|^{2}}{H^{2}} \langle \nu, e_{1} \rangle (x,t).
\end{equation*}
Choose 
\begin{equation*}
   \beta > \sup_{(x,t) \in \mathbb{S}^{n} \times [t_{0},t_{0} + \epsilon)} \frac{|A|^{2}}{H^{2}},
\end{equation*}
and define $f(x,t)=e^{\beta(t_{0}-t)} \langle e_{1}, \nu \rangle (x,t)$. Then
\begin{equation*}
f(x,t) \leq 0 \hspace{0.5cm} \text{  on    } \partial_{P} (U \times [t_{0},t_{0} + \epsilon)) = (U \times \{ t_{0} \}) \cup (\partial U \times [t_{0},t_{0} + \epsilon)),
\end{equation*}
and at a positive spacetime maximum $(x,t) \in U \times [t_{0},t_{0} + \epsilon)$ of $f$
\begin{equation} \label{gap}
    (\partial_{t} - \frac{1}{H^{2}} \Delta) f(x,t) = (\frac{|A|^{2}}{H^{2}} - \beta) f(x,t) < 0.
\end{equation}
This is a contradiction, and so $\langle e_{1}, \nu \rangle \leq 0$ on $U \times [t_{0}, t_{0} + \epsilon) \subset \mathbb{S}^{n} \times [0,T_{\max})$.

Now, take a sequence of times $t_{k}$ that decrease to $t_{0}$ such that $\lim_{k} a(t_{k})= \liminf_{t \searrow t_{0}} a(t)$. By our definition of $C^{-}_{t}$, over any small enough neighborhood $V \subset \mathbb{S}^{n}$ of the boundary $\partial C^{-}_{t} = \{ x \in \mathbb{S}^{n} | \tilde{u}(x,t)=a(t) \}$ we have $\langle e_{1}, \nu \rangle (x,t) > 0$ on $V \setminus \overline{C^{-}_{t}}$. This allows us to find a sequence of points $x_{k} \in \mathbb{S}^{n} \setminus \overline{C^{-}_{t_{k}}}$ with $\langle e_{1}, \nu \rangle(x_{k},t_{k}) > 0$ and $\lim_{k} \tilde{u}(x_{k},t_{k})= \liminf_{t \searrow t_{0}} a(t)$. After passing to a subsequence if necessary, we are left with a sequence satisfying

\begin{eqnarray}
    \langle e_{1}, \nu \rangle (x_{k},t_{k}) &>& 0, \\
    (x_{k},t_{k}) &\rightarrow& (x_{0},t_{0}) \hspace{0.3cm} \text{  S.T.  } \tilde{u}(x_{0},t_{0}) = \liminf_{t \searrow t_{0}} a(t),
\end{eqnarray}
 In view of \eqref{U}, the second line implies that $(x_{k},t_{k}) \in U \times [t_{0}, t_{0} + \epsilon)$ for large enough $k$, but then the first line contradicts non-positivity of $\langle e_{1}, \nu \rangle$ on this set. This completes the proof of the first part.
 
 For the second part, suppose no such neighborhood exists. Then there must be a sequence $(x_{n},t_{n}) \in \mathbb{S}^{n} \times [0,T_{\max})$ converging to $(x_{0},t_{0}) \in C^{-}_{t_{0}} \times \{ t_{0} \}$ with $t_{n} > t_{0}$ and $x_{n} \in \mathbb{S}^{n} \setminus C^{-}_{t_{n}}$. Then $a(t_{n}) \leq \tilde{u}(x_{n},t_{n})$, and the using convergence of $(x_{n},t_{n})$
 \begin{equation}
      \liminf_{n} a(t_{n}) \leq \lim_{n} \tilde{u}(x_{n},t_{n}) = \tilde{u}(x_{0},t_{0}) < a(t_{0}).
 \end{equation}
This contradicts the first part, so any sequence approaching $(x_{0},t_{0})$ from above in time lies in $C^{-}$.
\end{proof}

\begin{theorem}
Let $F_{0}: \mathbb{S}^{n} \rightarrow \mathbb{R}^{n+1}$ be an $H>0$, rotationally symmetric $C^{\infty}$ embedding, and $F: \mathbb{S}^{n} \times [0,T_{\max}) \rightarrow \mathbb{R}^{n+1}$ the corresponding maximal solution to \eqref{IMCF}. Then for any $(x_{0},t_{0}) \in \partial_{P} B \setminus \tilde{\partial_{P}} B \subset \mathbb{S}^{n} \times [0,T)$ with $t_{0}>0$, there is a sequence $(x_{n},t_{n}) \in L$ approaching $(x_{0},t_{0})$ with $t_{n} < t_{0}$. In particular, $B$ satisfies the hypothesis in the second part of \thref{noncyl}.
\end{theorem}

\begin{proof}
Take $(x_{0},t_{0}) \in \partial_{P} B \setminus \tilde{\partial}_{P} B$ with $t_{0} >0$. We need to show that $(x_{0},t_{0})$ is a limit point of $B \cap \{ t < t_{0} \}$.

Suppose $(x_{0},t_{0})$ is not a limit point of $B \cap \{ t < t_{0} \}$. Then there must exist a neighborhood $\mathcal{O}$ of $(x_{0},t_{0})$ such that $\mathcal{O} \cap \{t < t_{0}\} \subset C^{-}$ or $\mathcal{O} \cap \{ t < t_{0} \} \subset C^{+}$ (if $\mathcal{O}$ simultaneously intersects $C^{+}$ and $C^{-}$ it must also intersect $B$). Say W.L.0.G. $\mathcal{O} \cap \{ t < t_{0} \} \subset C^{-}$. We will show that $(x_{0},t_{0}) \in C^{-}_{t_{0}}$. Applying the previous theorem would imply that $(x_{0},t_{0})$ lies in the open set $C^{-}$, and this contradicts $(x_{0},t_{0}) \in \partial_{P} B$.

If $\mathcal{O} \cap \{ t < t_{0} \} \subset C^{-}$, there is an increasing sequence $t_{k} \nearrow t_{0}$ and corresponding points $x_{k} \in C^{-}_{t_{k}}$ that converge to $x_{0} \in \mathbb{S}^{n}$. Since $\tilde{u}(x_{k},t_{k}) \leq a(t_{k})$ by definition, we have $\tilde{u}(x_{0},t_{0}) \leq \limsup a(t_{k})$. For a fixed $ c \leq \tilde{u}(x_{0},t_{0})$, there are also points $y_{k} \in C^{-}_{t_{k}}$ with $\tilde{u}(y_{k},t_{k})=c$. 

Passing to a subsequence if neccessary, $(y_{k},t_{k})$ converge to some $(y,t_{0}) \in \mathbb{S}^{n} \times \{ t_{0} \}$ with $\tilde{u}(y,t_{0})= c$. Since $\langle e_{1}, \nu \rangle (y_{k},t_{k}) \leq 0$, we have $\langle e_{1}, \nu \rangle (y,t_{0}) \leq 0$. This means that $\langle e_{1}, \nu \rangle (x,t_{0}) \leq 0$ over the set $\{ x \in \mathbb{S}^{n} | \tilde{u}(x,t_{0}) \leq \tilde{u} (x_{0},t_{0}) \}$. Furthermore, $(x_{0},t_{0}) \not \in \partial C^{-}_{t_{0}}$ because $\partial C^{-}_{t_{0}} \subset \partial B_{t_{0}} \subset \tilde{\partial}_{P} B$. Altogether,

\begin{equation}
    \{ x \in \mathbb{S}^{n} | \tilde{u}(x,t_{0}) \leq \tilde{u}(x_{0},t_{0}) \} \subset C^{-}_{t_{0}}.
\end{equation}
In view of the previous theorem, there are neighborhoods $\mathcal{O}, \mathcal{O}'$ of $(x_{0},t_{0})$ in $\mathbb{S}^{n} \times [0,T_{\max})$ so that $\mathcal{O} \cap \{ t < t_{0}\} \subset C^{-}$ and $\mathcal{O}' \cap \{ t > t_{0} \} \subset C^{-}$. Then $(x_{0},t_{0}) \in \mathcal{O} \cap \mathcal{O'} \subset C^{-}$, but this contradicts $(x_{0},t_{0}) \in \partial_{P} B$. 
\end{proof}
\section{Evolution Equations}
\hspace{0.5cm} In this section, we determine evolution equations for any rotationally symmetric solution of \eqref{IMCF}. We present evolution equations for the mean curvature $H$, height function $u$, and the quantity $v= (\langle w, \nu(x,t) \rangle)^{-1}$. We also include the equation for support function $\langle \vec{F}-\vec{x}_{0}, \nu \rangle$ of the embedding taken with respect to a fixed point $x_{0} \in \mathbb{R}^{n+1}$, since this plays a role in the analysis of the caps.

\begin{theorem}[Evolution Equations for IMCF] \thlabel{evolution_equations}
Let $F_{0}: \mathbb{S}^{n} \rightarrow \mathbb{R}^{n+1}$ be a $C^{\infty}$, $H>0$ rotationally symmetric embedding, and $F: \mathbb{S}^{n} \times [0,T) \rightarrow \mathbb{R}^{n+1}$ the corresponding maximal solution to \eqref{IMCF}. Then for a fixed vector $\vec{e} \in \mathbb{R}^{n+1}$ and point $\vec{x}_{0} \in \mathbb{R}^{n+1}$, the following evolution equations hold

\begin{eqnarray} \label{evolution}
    \partial_{t} \nu &=& \frac{1}{H^{2}} \nabla H; \label{evolution1} \\
    (\partial_{t} -\frac{1}{H^{2}} \Delta) \langle \nu, \vec{e} \rangle &=& \frac{|A|^{2}}{H^{2}} \langle \nu, \vec{e} \rangle \label{evolution_e} \\
    (\partial_{t} - \frac{1}{H^{2}} \Delta) H &=& -\frac{|A|^{2}}{H^{2}}H -  2\frac{|\nabla H|^{2}}{H^{3}}; \label{evolution2} \\
    (\partial_{t} - \frac{1}{H^{2}} \Delta) H^{-1} &=& \frac{|A|^{2}}{H^{2}}H^{-1}; \label{evolution3} \\
  (\partial_{t} - \frac{1}{H^{2}} \Delta) \langle \vec{F} - \vec{x}_{0}, \nu \rangle &=& \frac{|A|^{2}}{H^{2}} \langle \vec{F} - \vec{x}_{0}, \nu \rangle; \label{evolution_extra} \\
    (\partial_{t} - \frac{1}{H^{2}} \Delta) u &=& \frac{2p}{H}u  - \frac{(n-1)p^{2}}{H^{2}} v^{2} u; \label{evolution4} \\
    (\partial_{t} - \frac{1}{H^{2}} \Delta) v &=& -\frac{|A|^{2}}{H^{2}} v  + \frac{(n-1)p^{2}}{H^{2}} v^{3} -  2\frac{|\nabla v|^{2}}{H^{2}v}; \label{evolution5}
\end{eqnarray}
\end{theorem}

\begin{proof}
The first five equations are available in \cite{huisken2}, section 1, and \cite{choi}, section 2. For equations \eqref{evolution4} and \eqref{evolution5}, the Laplacians of the quantities $u= \langle \vec{F}, w \rangle$ and $v= \langle \nu, w \rangle^{-1}$ are shown in, e.g. \cite{huisken1990}, section 5, and \cite{maria}, section 3, to be

\begin{eqnarray}
    \Delta_{N} u &=& \frac{n-1}{u} - \frac{H}{v}, \nonumber \\
    \Delta_{N} v &=& -v^{2} \langle \nabla H, w \rangle + |A|^{2} v - \frac{n-1}{u^{2}} v + 2 v^{-1} |\nabla v|^{2}. \nonumber
\end{eqnarray}

On the other hand, the time derivatives of these quantities may be computed as

\begin{eqnarray}
    \frac{\partial u}{\partial t} &=& \langle w, \frac{\partial F}{\partial t} \rangle = \frac{1}{Hv}, \nonumber \\
    \frac{\partial v}{\partial t} &=& -v^{2} \langle w, \frac{\partial \nu}{\partial t} \rangle = -v^{2} \langle w, \frac{\nabla H}{H^{2}} \rangle. \nonumber
\end{eqnarray}
Noting that $p=(uv)^{-1}$, \eqref{evolution4} and \eqref{evolution5} follow. 

\end{proof}
\section{A Priori Height Estimates}
In this section, we estimate the position vector $\vec{F}(x,t)$ of any rotationally symmetric solution to IMCF. We utilize a one-sided version of the well-known avoidance principle for MCF proven in Section 3 of \cite{me}. 

\begin{theorem}[One-Sided Avoidance Principle] \thlabel{avoidance}
Let $\{ N_{t} \}_{0 \leq t < T}$ and $\{ \tilde{N}_{t} \}_{0 \leq t < T}$ be two closed, connected solutions to \eqref{IMCF}. For each $t \in [0,T)$, let $E_{t} \subset \mathbb{R}^{n+1}$ and $\tilde{E}_{t} \subset \mathbb{R}^{n+1}$ be the bounded, open domains with $N_{t}=\partial E_{t}$ and $\tilde{N}_{t}=\partial \tilde{E}_{t}$. If $E_{0} \subset \tilde{E}_{0}$ then $E_{t} \subset \tilde{E}_{t}$, and $\text{dist}(N_{t},\tilde{N}_{t})$ is non-decreasing.
\end{theorem}

This immediately controls the width in the $\hat{e}_{1}$ direction of $N_{t}$.

\begin{proposition}[Width Estimate] \thlabel{height_bound}
Let $F_{0}: \mathbb{S}^{n} \rightarrow \mathbb{R}^{n+1}$ be a $C^{\infty}$, $H>0$, rotationally symmetric embedding, and $F: \mathbb{S}^{n} \times [0,T_{\max}) \rightarrow \mathbb{R}^{n+1}$ the corresponding solution to \eqref{IMCF}. Then for $\tilde{u}(x,t)=\langle \vec{F}(x,t), e_{1} \rangle$,

\begin{eqnarray}
    |\tilde{u}(x,t)| &\leq& (\max_{N_{0}} |\vec{F}|) e^{\frac{t}{n}}. \label{hor} 
\end{eqnarray}
\end{proposition}

\begin{proof}

$N_{0}$ is enclosed by a sphere of radius $\rho_{0}=\max_{N_{0}} |\vec{F}|$, so comparing with the corresponding spherical solution $\rho(t)=(\max_{N_{0}} |\vec{F}|)e^{\frac{t}{n}}$ using the one-sided avoidance principle yields \eqref{hor}.
\end{proof}
We can also control the height $u$ using Hamilton's trick.
\begin{proposition}[Height Estimate] \thlabel{height_growth}
Let $F_{0}: \mathbb{S}^{n} \rightarrow \mathbb{R}^{n+1}$ be a $C^{\infty}$, $H>0$, rotationally symmetric embedding, and $F: \mathbb{S}^{n} \times [0,T_{\max}) \rightarrow \mathbb{R}^{n+1}$ the corresponding solution to \eqref{IMCF}. Then for $u(x,t)=\langle \vec{F}(x,t), w \rangle$,
 \begin{equation} \label{vert}
 u(x,t) \leq (\max_{N_{0}} u)e^{\frac{t}{n-1}}.
 \end{equation}
\end{proposition}

\begin{proof}
Consider the function $f: [0,T) \rightarrow \mathbb{R}$ defined by $f(t)= \max_{x \in \mathbb{S}^{n}} e^{-\frac{t}{n-1}}u(x,t)$. According to Hamilton's trick, c.f. Section 2.1 of \cite{carlos}, $f$ is a locally Lipschitz function of time, and where differentiable satisfies

\begin{equation*}
    f'(t_{0})= \partial_{t} f(x_{0},t_{0})
\end{equation*}
where $(x_{0},t_{0}) \in \mathbb{S}^{n} \times [0,T)$ is any point maximizing $f$ at the time $t_{0}$. $\partial_{t} u(x,t)$ is simply the $\hat{w}$-component of the velocity vector $\frac{1}{H} \nu$ in $\mathbb{R}^{n+1}$, and so

\begin{equation} \label{height_deriv}
    \partial_{t} f(x,t)= e^{-\frac{t}{n-1}} \frac{\langle \nu, w \rangle}{H} - \frac{1}{n-1} f.
\end{equation}
At $(x_{0},t_{0})$ we have $\langle \nu, w \rangle(x_{0},t_{0}) =1$ and $k(x_{0},t_{0}) \geq 0$. This means $H(x_{0},t_{0}) \geq (n-1) u(x_{0},t_{0})^{-1}$. Plugging this into \eqref{height_deriv} yields

\begin{equation*}
    \partial_{t} f(x_{0},t_{0}) \leq 0.
\end{equation*}
Therefore, $f'(t_{0}) \leq 0$ where differentiable. For times $t_{1} < t_{2}$ in $[0,T)$ we use the Fundamental Theorem of Calculus to write

\begin{equation*}
    f(t_{2}) = f(t_{1}) + \int_{t_{1}}^{t_{2}} f'(t) dt \leq f(t_{1}).
\end{equation*}
The estimate follows.
\end{proof}
As a corollary of this, we also obtain a lower bound on $H$ over the boundary of $B_{t}$. We also note that the height $u$ is minimized over $B_{t}$ at an interior point.

\begin{corollary} \thlabel{cap_bound}
Let $F_{0}: \mathbb{S}^{n} \rightarrow \mathbb{R}^{n+1}$ be a $C^{\infty}$, $H>0$, rotationally symmetric embedding, and $F: \mathbb{S}^{n} \times [0,T_{\max}) \rightarrow \mathbb{R}^{n+1}$ the corresponding solution to \eqref{IMCF}. Then $k(x,t) \geq 0$ when $x \in \partial B_{t}$. As a consequence,

\begin{equation} \label{cap_H}
    H|_{\partial B_{t}} \geq (n-1) \frac{e^{\frac{-t}{n-1}}}{\max_{N_{0}} u}.
\end{equation}
Furthermore, $\min_{x \in B_{t}} u(x,t) < u|_{\partial B_{t}}$.
\end{corollary}

\begin{proof}
Whenever $F_{t}(x) \in \mathbb{R}^{n+1} \setminus X_{1}$,
\begin{equation*}
    \nabla \langle e_{1}, \nu \rangle =  \hat{v}_{1} (\langle e_{1}, \nu \rangle) \hat{v}_{1} = \langle e_{1}, \nabla_{\hat{v}_{1}} \nu \rangle \hat{v}_{1} = k \langle w, v \rangle \hat{v}_{1},
\end{equation*}
where once again $\hat{v}_{1}$ is defined in \eqref{v_1}. Now, if $(x,t) \in \partial B_{t}= \partial C^{+}_{t} \cup \partial C^{-}_{t}$, then $\hat{v}_{1}=\hat{e}_{1}$ and the $\hat{v}_{1}$ component of $\nabla \langle e_{1}, \nu \rangle$ is non-negative, and so $k(x,t) \geq 0$. Then $H(x,t) \geq (n-1) p(x,t) = (n-1) u^{-1}(x,t)$ here, and since $u \leq e^{\frac{t}{n-1}} \max_{N_{0}} u$ from \eqref{vert} the first conclusion follows.

For the second part, once again for a small enough neighborhood $V$ of $\langle e_{1}, \nu \rangle$, $\langle e_{1}, \nu \rangle > 0$ on $V \setminus \overline{C^{-}_{t}}$ (resp. $<0$ over $C^{+}_{t}$). Taking an integral curve of $\hat{v}_{1}$ from $y \in C^{-}_{t}$ to $x \in V \setminus \overline{C^{-}_{t}}$ and using equation \eqref{height_grad} for $\nabla u$ yields the conclusion.
\end{proof}

\section{The Bridge Region}
We first consider the region $B \subset \mathbb{S}^{n} \times [0,T)$, as the geometry of this domain allows us to apply the non-cylindrical maximum principle. Many of the estimates derived in this section apply generally for any mean-convex, rotationally symmetric embedded sphere, but a crucial sharp bound on the quantity $v$ from \eqref{v}, which roughly measures how ``narrow" the necks are, only applies for admissible data. Once again, we must find a uniform-in-time bound on the flow speed over $B$. We begin by estimating the principal curvature $p$ of rotation.

\begin{theorem} [Rotational Curvature Estimates] \thlabel{rot}
Let $F_{0}: \mathbb{S}^{n} \rightarrow \mathbb{R}^{n+1}$ be a $C^{\infty}$, $H>0$, rotationally symmetric embedding, and $F: \mathbb{S}^{n} \times [0,T_{\max}) \rightarrow \mathbb{R}^{n+1}$ the corresponding solution to \eqref{IMCF}. Then the principal curvature $p=(uv)^{-1}$ obeys the estimates

\begin{equation} \label{rotational_curvature}
    e^{-\frac{t}{n-1}} \min_{\overline{B}_{0}} p \leq p (x,t) \leq e^{-\frac{t}{n-1}} \max_{\overline{B}_{0}} p 
\end{equation}
over $B= \text{Int}(\cup_{0 \leq t < T} B_{t} \times \{ t \})$. In particular, $\frac{\max_{\overline{B}_{t}} p}{\min_{\overline{B}_{t}} p}$ is a non-increasing function of time.
\end{theorem}

\begin{proof}
Combining equations \eqref{evolution4} and \eqref{evolution5} yields the following for $p=(uv)^{-1}$:


\begin{eqnarray*}
    (\partial_{t} - \frac{1}{H^{2}} \Delta) p &=& u^{-1}(\frac{|A|^{2}}{H^{2}} v^{-1} - \frac{(n-1)p^{2}}{H^{2}} v^{2}) + \\ 
    & &  v^{-1} (-\frac{2p}{H}u^{-1} + \frac{(n-1)p^{2}}{H^{2}} v^{2} u^{-1} - 2\frac{u^{-3}}{H^{2}} |\nabla u|^{2}) - \\
    & &  2\frac{(uv)^{-2}}{H^{2}} \langle \nabla u, \nabla v \rangle \\
    &=& (\frac{|A|^{2}}{H^{2}} - \frac{2p}{H}) p + \frac{2}{H^{2}u} \langle \nabla p, \nabla u \rangle.
\end{eqnarray*}
The function $f(x,t)=e^{\frac{t}{n-1}} p(x,t)$ then satisfies

\begin{equation} \label{evolution_f}
    (\partial_{t} - \frac{1}{H^{2}} \Delta) f = (\frac{|A|^{2}}{H^{2}} - \frac{2p}{H} + \frac{1}{n-1}) f + \frac{2}{H^{2}u} \langle \nabla f, \nabla v \rangle.
\end{equation}
We have at a spacetime maximum or minimum $(x_{0},t_{0})$ of $f$ in $B$ that $\nabla p(x_{0},t_{0}) = 0$. According to the formula \eqref{p_grad} for $\nabla p$, critical points of $p$ are characterized by either $\langle \nu, e_{1} \rangle =0$ or $k=p$. We consider these cases separately.

\textit{Case I: $\langle \nu, e_{1} \rangle =0$}: If $(x_{0},t_{0})$ is a minimum of $p$, then over a sufficiently small neighborhood of $(x_{0},t_{0})$

\begin{equation*}
    u^{-1}(x_{0},t_{0})=p(x_{0},t_{0}) \leq p(x,t) \leq u^{-1}(x,t).
\end{equation*}
So $(x_{0},t_{0})$ is local maximum of the height function over $B_{t_{0}}$. This guarantees that $k \geq 0$ and hence $H \geq (n-1) p$ at this point, and since $1$ is an absolute minimum of the function $v$ we know $\partial_{t} v(x_{0},t_{0}) =0$. Altogether,

\begin{eqnarray*}
    \partial_{t} p(x_{0},t_{0}) &=& -\frac{1}{v^{2}(x_{0},t_{0})u(x_{0},t_{0})} \partial_{t} v(x_{0},t_{0}) - \frac{1}{v(x_{0},t_{0}) u^{-2}(x_{0},t_{0})}\partial_{t} u (x_{0},t_{0}) \nonumber \\
    &=& -\frac{p^{2}(x_{0},t_{0})}{H(x_{0},t_{0})} \geq -\frac{1}{n-1} p(x_{0},t_{0}).
\end{eqnarray*}
This implies $\partial_{t} f(x_{0},t_{0}) \geq 0$ at a minimum. On the other hand, the global maximum of $p$ on $B_{t_{0}}$ corresponds to the global minimum of $u$, which by \thref{cap_bound} this occurs at an interior point of $B_{t_{0}}$. At this point

\begin{equation*}
    \partial_{t} p(x_{0},t_{0}) = -\frac{p^{2}(x_{0},t_{0})}{H(x_{0},t_{0})} \leq -\frac{1}{n-1} p(x_{0},t_{0}),
\end{equation*}
where the inequality follows from the fact that $k(x_{0},t_{0}) \leq 0$. Therefore $\partial_{t} f(x_{0},t_{0}) \leq 0$ at a maximum.

\textit{Case II: $k=p$:} Since the maximum of $p$ on $B_{t_{0}}$ occurs at an interior minimum of $u$ which is covered by Case I, $(x_{0},t_{0})$ is a global minimum of $f$. $N_{t_{0}}$ is umbilic at this point, so $H^{2}=n|A|^{2}$ and each principal curvature equals $p$. Thus \eqref{evolution_f} becomes

\begin{equation*}
    (\partial_{t} - \frac{1}{H^{2}} \Delta) f(x_{0}, t_{0}) = (-\frac{1}{n} + \frac{1}{n-1}) f(x_{0},t_{0}) \geq 0.
\end{equation*}
 Altogether, $\partial_{t} f (x_{0},t_{0}) \leq 0$ at any spacetime maximum (resp. $\geq 0$ at any minimum) in $L$, and the non-cylindrical maximum principle yields

\begin{equation*}
   \inf_{\tilde{\partial}_{P} B} f \leq f(x,t) \leq \sup_{\tilde{\partial}_{P} B} f.
\end{equation*}
In fact, $p|_{\partial B_{t}} = u^{-1}|_{\partial B_{t}} \geq e^{\frac{-t}{n-1}} (\max_{B_{0}} u)^{-1}$ from \thref{height_growth}, and so $f|_{\partial B_{t}} \geq \min_{B_{0}} f$. Likewise, the maximum of $f$ at the time $t$ cannot occur on $\partial B_{t}$ by \thref{cap_bound}, so the supremum and infimum over the reduced parabolic boundary happen at $t=0$. Altogether,

\begin{equation*}
    e^{-\frac{t}{n-1}} \inf_{B_{0}} p \leq p(x,t) \leq e^{-\frac{t}{n-1}} \sup_{B_{0}} p, 
\end{equation*}
and so

\begin{equation}
    h(t)=\frac{\max_{\overline{B_{t}}} p}{\min_{\overline{B_{t}}} p}
\end{equation}
is a non-increasing function of time.

\end{proof}

\begin{remark}
The umbilicity of the $N_{t}$ at critical points of $p$ makes the reaction terms in its evolution equation much more tractable compared to the evolution equation under MCF for the same quantity.
\end{remark}

\thref{rot} provides a sharp interior gradient estimate, as one can show that the gradient-like quantity $v$ is bounded by the ratio of the highest and lowest values of $p$ at time $t$. If we use the admissibility condition and an integration trick from \cite{angenent}, we can obtain for admissible data that $v$ is specifically bounded away from $\sqrt{n}$ over $L \subset \mathbb{S}^{n} \times [0,T)$.
\begin{corollary}
Let $F_{0}: \mathbb{S}^{n} \rightarrow \mathbb{R}^{n+1}$ be an admissible rotationally symmetric embedding, and $F: \mathbb{S}^{n} \times [0,T_{\max}) \rightarrow \mathbb{R}^{n+1}$ the corresponding solution to \eqref{IMCF}. Then

\begin{equation} \label{v_est}
    \max_{\overline{B}} v < \sqrt{n}
\end{equation}
\end{corollary}
\begin{proof}
To prove this statement we consider two cases separately. Either $u(x,t) \geq n^{\frac{1}{2(n-1)}} \min_{B_{t}} u$ or $u(x,t) < n^{\frac{1}{2(n-1)}} \min_{B_{t}} u$.

\textit{Case I: $u(x,t) \geq n^{\frac{1}{2(n-1)}} \min_{B_{t}} u$:} We know $(\max_{B_{t}} p)^{-1}= \min_{B_{t}} u$. Then 

\begin{equation*}
    v(x,t) \leq \frac{\max_{B_{t}} uv}{u(x,t)} \leq \frac{1}{n^{\frac{1}{2(n-1)}}} \frac{\max_{B_{t}} p}{\min_{B_{t}} p} \leq n^{-\frac{1}{2(n-1)}} \frac{\max_{B_{0}} p}{\min_{B_{0}} p} < \sqrt{n},
\end{equation*}
where we used admissibility and that $\frac{\max_{B_{t}} p}{\min_{B_{t}} p}$ is a nonincreasing function of time.

\textit{Case II: $u(x,t) < n^{\frac{1}{2(n-1)}} \min_{B_{t}} u$:} Recall equations \eqref{w_grad} and \eqref{height_grad},

\begin{eqnarray*}
    \nabla u &=& - \langle e_{1}, \nu \rangle \hat{v}_{1}\\
    \nabla \langle w, \nu \rangle &=& - \langle e_{1}, \nu \rangle k \hat{v}_{1}.
\end{eqnarray*}
Relating these, we find the following equation for the gradient of $\ln(v(x,t))$,

\begin{equation}
    \nabla \ln (v) = -kv \nabla u.
\end{equation}
Consider the image point $F_{t}(x)$ on $N_{t}$ of $(x,t)$. If $\langle \nabla u, \hat{v}_{1} (x,t) \rangle < 0$, take the integral curve $\gamma$ of the vector field $\hat{v}_{1}$. Let $s_{0} >0$ be the first parameter value with $\langle \nabla u(\gamma(s)), \hat{v}_{1} \rangle =0$, and say $\gamma(s_{0})=F_{t}(y)$. Using $(n-1) p  + k > 0$, $\ln v(y,t) = 0$, and $\langle \nabla u(\gamma(s)), \hat{v}_{1} \rangle < 0$ for $s \in [0,s_{0})$, we find

\begin{eqnarray}
    -\ln v(x,t) &=& \int_{0}^{s_{0}} \langle \nabla \ln v (x,t), \hat{v}_{1} \rangle ds = \int_{0}^{s_{0}} -kv \langle \nabla u, \hat{v}_{1} \rangle ds \nonumber \\
    &>& \int_{0}^{s_{0}} \frac{(n-1)}{u} \langle \nabla u, \hat{v}_{1} \rangle ds = \int_{0}^{s_{0}} (n-1) \langle \nabla \ln u, \hat{v}_{1} \rangle ds \\
    &=& (n-1) \ln (\frac{u(y,t)}{u(x,t)}) \nonumber.
\end{eqnarray}
If $\langle \nabla u, \hat{v}_{1} (x,t) \rangle > 0$, we just take the corresponding $s_{0} < 0$ and the integral curvature $\gamma$ over $(s_{0},0]$ and obtain the same result. So 

\begin{equation*}
    v(x,t) \leq (\frac{u(x,t)}{u(y,t)})^{(n-1)},
\end{equation*}
for some critical point $(y,t)$ of $u$. In particular,

\begin{equation}
    v(x,t) \leq (\frac{u(x,t)}{\min_{B_{t}} u})^{n-1} < \sqrt{n}.
\end{equation}
\end{proof}
We are now ready to estimate $H^{-1}$ over the region $B$. \thref{cap_bound} ensures that $H$ is bounded below over $\partial B_{t}$ and hence the entire reduced parabolic boundary of $B$. Due to the positive term of evolution equation \eqref{evolution3} for $H^{-1}$, one seeks another well-behaved quantity to combine with the flow speed in order to use a maximum principle.

Equation \eqref{evolution5} suggests that $v$ is the most natural quantity to combine with the speed function, but due to an extra positive term one finds $(\partial_{t} - \frac{1}{H^{2}} \Delta) \frac{v}{H} \leq \frac{n-1}{H^{2}u^{2}} \frac{v}{H}$ at an interior maximum, meaning the RHS cannot be immediately controlled. In view of the estimate \eqref{v_est} on $v$, one can compensate for this term using the function $\varphi(r)= \frac{r}{1-\lambda r}$ from the proof of Theorem 3.1 in \cite{ecker2} (see also Proposition 5 in \cite{maria} and Theorem A.5 in \cite{choi}). The lower bound on $p$ from \thref{rot} will also be important in the proof.

The time has come. Execute Theorem 6.6.

\begin{theorem}[Speed Estimate over $B$] \thlabel{speed}
Let $F_{0}: \mathbb{S}^{n} \rightarrow \mathbb{R}^{n+1}$ be an admissible rotationally symmetric embedding, and $F: \mathbb{S}^{n} \times [0,T_{\max}) \rightarrow \mathbb{R}^{n+1}$ the corresponding solution to \eqref{IMCF}. If $T< \infty$ there is a constant $C=C(T, N_{0}, n) < \infty$ so that

\begin{equation}
    \sup_{B} \frac{1}{H} \leq C.
\end{equation}

\end{theorem}
\begin{proof}



Consider the function $\varphi(v)= \frac{v}{(1- \lambda v)}$ for $\lambda$ to be chosen later. From equation \eqref{evolution5}, one finds

\begin{eqnarray*}
    (\partial_{t} - \frac{1}{H^{2}} \Delta) \varphi(v) &=& - \frac{|A|^{2}}{H^{2}} \varphi'(v) v + \frac{(n-1)p^{2}}{H^{2}} \varphi'(v) v^{3} - (2 \frac{\varphi'(v)}{v} + \varphi''(v)) \frac{|\nabla v|^{2}}{H^{2}}.
\end{eqnarray*}
Define $g=H^{-1} \varphi(v)$. Using the relations

\begin{eqnarray*}
v \varphi'(v) - \varphi(v)&=& -\lambda |\varphi(v)|^{2}, \\ \varphi'(v) v^{2} &=& |\varphi(v)|^{2}, \\
2\frac{\varphi'(v)}{v} + \varphi''(v) &=& 2 \frac{|\varphi'(v)|^{2}}{\varphi(v)},
\end{eqnarray*}
we compute

\begin{eqnarray}
    (\partial - \frac{1}{H^{2}} \Delta) g &=& -|A|^{2} H^{-3} \varphi'(v) v + (n-1)p^{2} H^{-3} \varphi'(v) v^{3}  - (2 \frac{\varphi'(v)}{v} + \varphi''(v)) H^{-3} |\nabla v|^{2} \nonumber \\ & &
     + |A|^{2} H^{-3} \varphi(v)  - 2H^{-2} \langle \nabla \varphi(v), \nabla H^{-1} \rangle \nonumber \\
    &=&(-\lambda |A|^{2} + v(n-1)p^{2})H^{-1}g^{2}  - 2 \frac{|\varphi'(v)|^{2}}{\varphi(v)} H^{-3} |\nabla v|^{2} \label{evolution_g} \\ & &- 2H^{-2} \langle \nabla \varphi(v), \nabla H^{-1} \rangle \nonumber \\
    &=& (-\lambda |A|^{2} + v(n-1)p^{2})H^{-1}g^{2} - 2(H^{2}\varphi(v))^{-1} \langle \nabla \varphi(v), \nabla g \rangle. \nonumber
\end{eqnarray}
As $\max_{\overline{B}} v < \sqrt{n}$ for admissible $N_{0}$, let $\frac{1}{\sqrt{n}} < \lambda < \frac{1}{\max_{\overline{B}} v}$. Since the corresponding $\varphi(v)$ is bounded over $B$, $H=(n-1)p + k$ must be near zero when $g$ is large enough. By the lower bound on $p$ of \thref{rot}, $-k \rightarrow (n-1)p$ and therefore $|A|^{2} \rightarrow n(n-1) p^{2}$ as $H \rightarrow 0$. Thus $\lambda |A|^{2} \geq \sqrt{n} (n-1)p^{2}$ for sufficiently large $g$, and so once again in view of the bound on $v$ the first term in the last line of \eqref{evolution_g} will be non-positive when this happens.

Since this term is clearly bounded for small $g$, take $\tilde{g}= g - Ct$ for some constant $C=C(n,N_{0})$ chosen so that $(-\lambda |A|^{2} + v(n-1)p^{2})H^{-1}g^{2} - C$ is strictly non-positive. $\tilde{g}$ satisfies

\begin{equation*}
    (\partial_{t} - \frac{1}{H^{2}} \Delta) \tilde{g} \leq \langle \eta, \nabla \tilde{g} \rangle
\end{equation*}
 for $\eta = 2(H^{2}\varphi(v))^{-1} \nabla \varphi(v)$ over $B \subset \mathbb{S}^{n} \times [0,T)$. As mean curvature is bounded below by $(n-1) e^{-\frac{T}{n-1}} (\min_{B_{0}} u)^{-1}$ over $\tilde{\partial}_{P} B$ according to \thref{height_growth}, $\tilde{g} \leq \sup_{\tilde{\partial}_{P} B} \tilde{g} \leq C(N_{0},n) e^{\frac{T}{n-1}}$ by the non-cylindrical maximum principle. This bounds the growth of $H^{-1}= (\tilde{g} + Ct)(\varphi(v))^{-1}$ to linear plus exponential, i.e.

\begin{equation*}
    \sup_{B} H^{-1} \leq C_{1} e^{\frac{T}{n-1}} + C_{2} T,
\end{equation*}
for constants $C_{1}, C_{2}$.
\end{proof}

\begin{remark}
The second condition in \thref{admissible} is necessary to ensure $\sup_{B} v< \sqrt{n}$, allowing us to define $g$ in such a way that it is controlled using the non-cylindrical maximum principle. There is still a time-independent bound on $v$ over this region for non-admissible data, but not by $\sqrt{n}$. It is unclear whether $H^{-1}$ is bounded over the bridge for non-admissible data.
\end{remark}


\section{The Cap Region}
As a result of the previous section, $H^{-1}$ is uniformly controlled over $\partial_{P} B =\partial_{P} C^{+} \cup \partial_{P} C^{-}$ for admisible initial data. This means that we apply the first part of the non-cylindrical maximum principle in order to control this quantity over $C^{+}$ and $C^{-}$. The maximum principle used in this section applies for any smooth, $H>0$, rotationally symmetric embedded sphere, but it is only for admissible data that we can control the relevant quantity on the parabolic boundary.

Like in the last section, we require a positive, bounded quantity to combine with the flow speed in order to obtain a useful evolution equation. $\langle \nu, e_{1} \rangle$ is non-negative over the right cap $C^{+}_{t}$ (Respectively non-positive over $C^{-}_{t}$) according to \thref{subset}. This allows us to fix an appropriate point on the axis $X_{1}$ so that the support function of the flow surfaces with respect to this point is strictly positive over one of the caps.

\begin{definition}
Let $F_{0}: \mathbb{S}^{n} \rightarrow \mathbb{R}^{n+1}$ be a $C^{\infty}$, $H>0$, rotationally symmetric embedding, and $F: \mathbb{S}^{n} \times [0,T_{\max}) \rightarrow \mathbb{R}^{n+1}$ the corresponding solution to \eqref{IMCF}. For a fixed time interval $[0,T)$, $T < +\infty$, consider the point $x_{0} \in X_{1} \subset \mathbb{R}^{n+1}$ given by

\begin{equation*}
    x_{0}=(\max_{N_{0}} |\vec{F}| e^{\frac{T}{n}}, 0, \dots, 0).
\end{equation*}
The \textbf{right support function $\theta_{+}: \mathbb{S}^{n} \times [0,T) \rightarrow \mathbb{R}$} and \textbf{left support function $\theta_{-}: \mathbb{S}^{n} \times [0,T) \rightarrow \mathbb{R}$} are defined as 

\begin{eqnarray*}
    \theta_{+}(x,t)&=& \langle \vec{F}(x,t) + \vec{x}_{0}, \nu \rangle, \\
    \theta_{-}(x,t)&=& \langle \vec{F}(x,t) - \vec{x}_{0}, \nu \rangle.
\end{eqnarray*}
\end{definition}
This particular choice of $x_{0}$ ensures that $\theta_{\pm}$ remains positive over each respective cap, see Figure \ref{support}.

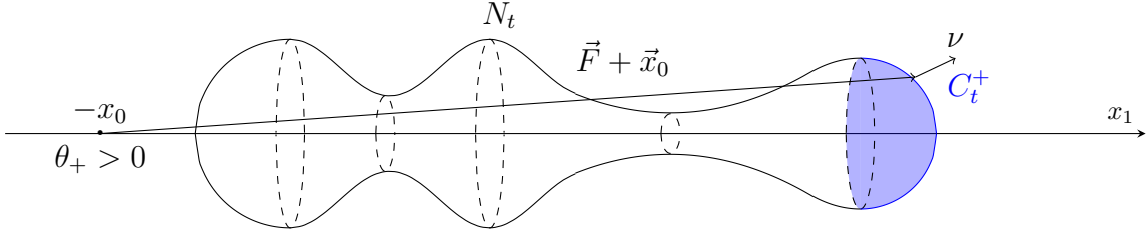
\begin{figure}

    \begin{center}
        \textbf{The Left and Right Support Functions}
        \vspace{0.3cm}
    \end{center}
    \centering
    \begin{tikzpicture}[xscale=2.5, yscale=2.5]
    \begin{axis}[ xlabel=\footnotesize{$x_{1}$}, y=0.5cm, x=0.5cm,
          xmax=6, xmin=-6, ymax=1.5, ymin=-1.2,
          axis lines=middle,
          restrict y to domain=-2:2,
           xtick = {0},
        ytick= {0},
        y axis line style= {draw= none},
        yticklabels= {0, $u_{\min}(0)$}
          enlargelimits]
        \addplot[ domain=-4:-3, smooth] {(1 - (x+3)*(x+3))^(0.5)};
        \addplot[ domain=3:3.8, smooth, blue, name path=g] {(0.8*0.8 - (x-3)*(x-3))^(0.5)};
        \addplot[ domain=-3:0, smooth] {0.3*cos(deg(3*(x+3))) + 0.7};
        \addplot[ domain=-2.1:-1.9, smooth, dashed] {0.4*(1-100*(x+2)^(2))^(0.5)};
        \addplot[ domain=-2.1:-1.9, smooth, dashed] {-0.4*(1-100*(x+2)^(2))^(0.5)};
        
         \addplot[ domain=0.9:1.1, smooth, dashed] {0.21*(1-100*(x-1)^(2))^(0.5)};
        \addplot[ domain=0.9:1.1, smooth, dashed] {-0.21*(1-100*(x-1)^(2))^(0.5)};
         \addplot[ domain=-2.85:-3.15, smooth, dashed] {(1-44.44*(x+3)^(2))^(0.5)};
           \addplot[ domain=-2.85:-3.15, smooth, dashed] {-(1-44.44*(x+3)^(2))^(0.5)};
        \addplot[ domain=-1.05:-0.75, smooth, dashed] {(1-44.44*(x+0.9)^(2))^(0.5)};
                \addplot[ domain=-1.05:-0.75, smooth, dashed] {-(1-44.44*(x+0.9)^(2))^(0.5)};
         \addplot[ domain=2.85:3, smooth, dashed, name path=f] {0.8*(1-44.44*(x-3)^(2))^(0.5)};
         \addplot[ domain=3:3.15, smooth, dashed] {0.8*(1-44.44*(x-3)^(2))^(0.5)};
             \addplot[ domain=2.85:3, smooth, dashed, name path=-f] {-0.8*(1-44.44*(x-3)^(2))^(0.5)};
         \addplot[ domain=3:3.15, smooth, dashed] {-0.8*(1-44.44*(x-3)^(2))^(0.5)};
         
      \addplot[blue, opacity=0.3] fill between[of =f and -f];
      
        \draw (-0.8,1) node[anchor=south]{$N_{t}$};
        \draw (3.8, 0.5) node[anchor=west]{$\color{blue} C_{t}^{+}$};
        \draw (-5,0) node{\Large{$\cdot$}};
        \draw (-5,0) node[anchor=south]{$-x_{0}$};
        \draw[->] (-5,0) -- (3.58,0.6);
        \draw[->] (3.58,0.6) -- (4,0.8);
        \draw (4,0.8) node[anchor=south]{$\nu$};
        \draw (0.5,0.8) node{$\vec{F}+ \vec{x}_{0}$};
        \draw (-5,0) node[anchor=north]{$\theta_{+} > 0$};
        \addplot[ domain=0:2.5, smooth] {0.2*(x-1)^(2) + 0.22};
        \addplot[ domain=0:2.5, smooth] {-(0.2*(x-1)^(2) + 0.22)};
         \addplot[ domain=2.5:3, smooth] {-0.55*(x-3)^(2) + 0.8};
          \addplot[ domain=2.5:3, smooth] {-(-0.55*(x-3)^(2) + 0.8)};
        \addplot[ domain=-4:-3, smooth] {-(1 - (x+3)*(x+3))^(0.5)};
        \addplot[ domain=3:3.8, smooth, blue, name path=-g] {-(0.8*0.8 - (x-3)*(x-3))^(0.5)};
        \addplot[ domain=-3:0, smooth] {-0.3*cos(deg(3*(x+3))) - 0.7};
          \addplot[blue, opacity=0.3] fill between[of =g and -g];
        \end{axis}
    \end{tikzpicture}

    \caption{By picking a point on the axis away from the flow surfaces for $t \in [0,T)$, we ensure the support function is positive over the right cap.}
    \label{support}
\end{figure}

\begin{proposition}
For any $t \in [0,T)$, the functions $\theta_{+}$ and $\theta_{-}$ are positive over $\overline{C^{+}} \subset \mathbb{S}^{n} \times [0,T)$ and $\overline{C^{-}} \subset \mathbb{S}^{n} \times [0,T)$, respectively.
\end{proposition}
\begin{proof}
We prove this for right cap first. The shifted $x_{1}$ coordinate

\begin{equation*}
    \tilde{u}(x,t)= \langle \vec{F}(x,t) + \vec{x}_{0}, e_{1} \rangle
\end{equation*}
must be strictly positive over $\mathbb{S}^{n} \times [0,T)$ in view of the width estimate from \thref{height_bound}. On the other hand, $\langle e_{1}, \nu \rangle$ is non-negative over $C^{+}$, and can only equal $0$ where $u$ and $\langle w, \nu \rangle$ are each positive. Also, $\langle e_{1}, \nu \rangle=1$ where $u=0$ on $C^{+}$. Recalling equation \eqref{supp} for the support function, we have

\begin{equation*}
       \theta_{+}(x,t) = \tilde{u} \langle \nu, e_{1} \rangle + u \langle w, \nu \rangle \geq c > 0
\end{equation*}
over $C^{+}$. The first term is also non-negative for $\theta^{-}$ over $C^{-}$, and can only equal $0$ when the second term is bounded below.
\end{proof}

We now consider the functions $f_{\pm}(x,t)= (\theta_{\pm} H)^{-1}$. $f_{+}$ and $f_{-}$ are well-defined and positive over $C^{+}$ and $C^{-}$, respectively. According to \thref{evolution_equations}, $\theta_{\pm}$ and $H^{-1}$ satisfy the same evolution equation. Thus the maximum principle applied to $f_{+}$ over $C^{+}$ (Resp. $f_{-}$ over $C^{-}$) yields an upper bound on $H^{-1}$.

\begin{theorem}[Speed Estimate over $C^{\pm}$] \thlabel{cap_speed}
Let $F_{0}: \mathbb{S}^{n} \rightarrow \mathbb{R}^{n+1}$ be a $C^{\infty}$, $H>0$, rotationally symmetric embedding, and $F: \mathbb{S}^{n} \times [0,T_{\max}) \rightarrow \mathbb{R}^{n+1}$ the corresponding solution to \eqref{IMCF}. For the functions $f_{+}: C^{+} \rightarrow \mathbb{R}$ and $f_{-}: C^{-} \rightarrow \mathbb{R}$ defined by $f_{+}(x,t)= (\theta_{+}(x,t) H(x,t))^{-1}$ and $f_{-}(x,t)=(\theta_{-}(x,t) H(x,t))^{-1}$,

\begin{equation*}
    \sup_{C^{\pm}} f_{\pm} = \sup_{\partial_{P} C^{\pm}} f_{\pm}. 
\end{equation*}
Furthermore, if $F_{0}$ is admissible then $\sup_{\partial_{P} C^{+}} f_{+} \leq c_{1}$ and $\sup_{\partial_{P} C^{-}} f_{-} \leq c_{2}$ for some constants $c_{1}=c_{1}(N_{0},T)$ and $c_{2}=c_{2}(N_{0},T)$. In this case, if $T< +\infty$ there is some constant $C=C(N_{0},T)$ so that

\begin{equation}
    \sup_{C^{+} \cup C^{-}} \frac{1}{H} \leq C.
\end{equation}
\end{theorem}

\begin{proof}
We present the proof for $f_{+}$, finding the evolution equation first. From equations \eqref{evolution2} and \eqref{evolution_extra} of \thref{evolution_equations}, one can compute

\begin{eqnarray*}
    (\partial_{t} - \frac{1}{H^{2}} \Delta) f_{+} &=& -\frac{2}{H^{2}} f_{+}^{-1} |\nabla f_{+}|^{2} - \frac{2}{H^{3}} \langle \nabla H, \nabla f_{+} \rangle.
\end{eqnarray*}
Calling $\eta= -\frac{2}{H^{2}} f_{+}^{-1} \nabla f_{+} - \frac{2}{H^{3}} \nabla H$, the maximum principle implies

\begin{equation*}
    \sup_{C^{+}} f_{+} \leq \max_{\partial_{P}C^{+}} f_{+}.
\end{equation*}

For the second part of the theorem, we have that $\theta_{+}$ is uniformly bounded away from $0$ over $C^{+}$, and as $\partial_{P} C^{+} \subset \partial_{P} B$, $\sup_{\partial_{P} C^{+}} H^{-1} \leq C(T,N_{0})$ due to \thref{speed}. This yields $\sup_{\partial_{P} C^{+}} f_{+} \leq C(T,N_{0})$, and in turn

\begin{equation*}
    \sup_{C^{+}} H^{-1} \leq C(T,N_{0}).
\end{equation*}
The proof is the same for $C^{-}$.
\end{proof}

\section{Global Existence and Convergence for Admissible Data}
Huisken and Ilmanen show in \cite{huisken2} that as long as the flow speed remains bounded near a time $T$, one may continue the flow past this time. Thus \thref{speed} establishes that the solution to IMCF starting from an admissible initial surface exists for all time.

\begin{corollary}[Global Existence and Convergence for Admissible Data] \thlabel{key_theorem}
Let $F_{0}: \mathbb{S}^{n} \rightarrow \mathbb{R}^{n+1}$ be an admissible rotationally symmetric embedding, and $F: \mathbb{S}^{n} \times [0,T_{\max}) \rightarrow \mathbb{R}^{n+1}$ the corresponding solution to \eqref{IMCF}. Then $T_{\max} = + \infty$. 
\end{corollary}

\begin{proof}
Take $T < + \infty$, and consider the solution to \eqref{IMCF} over $\mathbb{S}^{n} \times [0,T)$. We have from \thref{speed} that $\sup_{B} \frac{1}{H} \leq C_{1}(T,N_{0})$ and from \thref{cap_speed} that $\sup_{C^{+} \cup C^{-}} \frac{1}{H} \leq C_{2}(T,N_{0})$, so altogether $\frac{1}{H} \leq \max\{ C_{1}, C_{2} \} $ over $\mathbb{S}^{n} \times [0,T)$. According to Corollary 2.3 in \cite{huisken2}, we obtain a smooth, $H>0$ limit surface $N_{T}$ at the time $T$, and hence by parabolicity of \eqref{IMCF} there exists a solution in short time starting from $N_{T}$. Conclude by continuation that $T_{\max} = + \infty$.
\end{proof}

We can obtain a stronger profile of the flow using \cite{me}.

\begin{proof}[Proof of \thref{long_time}]
Since $N_{t}$ exists for all time and remains embedded under the flow, Theorem 4 from \cite{me} reveals $N_{t}$ is star-shaped by the time $t_{*}=n \text{log}(R^{-1} \text{diam} (N_{0}))$, where $R$ is the inradius of $N_{0}$. Theorem 0.1 in \cite{claus} then implies $C^{2}$ convergence to spheres for some choice of parametrizations $\tilde{F}_{t}$ of $\tilde{N_{t}}$, and Theorem 0.1 in \cite{Urbas} upgrades the strength of convergence to $C^{\infty}$.
\end{proof}

\section{Applications}
In this section, we discuss applications of a long-time solution to \eqref{IMCF}. One of these applications is a proof of the Minkowski inequality for certain non-convex domains, and the key to this is a monotonicity formula along IMCF first noted in \cite{Guan2009TheQI}. For the convenience of the reader, we briefly compute this formula here.

\begin{proof}[Proof of \thref{minkowski}]
Using the evolution equation \eqref{evolution3} for $H$ under IMCF and the variation formula $\partial_{t} d \mu = \langle \partial_{t} \vec{F}, H \nu \rangle d \mu= d \mu$ for the measure, we compute

\begin{eqnarray*}
    \partial_{t} \int_{N_{t}} H d \mu &=&\int_{N_{t}} (\frac{1}{H^{2}} \Delta H - 2 \frac{|\nabla H|^{2}}{H^{3}} - \frac{|A|^{2}}{H^{2}} H) d\mu + \int_{N_{t}} H d\mu \\
    &=& \int_{N_{t}} (1- \frac{|A|^{2}}{H^{2}}) H d \mu \leq \int_{N_{t}} \frac{n-1}{n} H d\mu,
\end{eqnarray*}
where the inequality follows from $|A|^{2} \geq \frac{1}{n} H^{2}$ and is strict unless each $N_{t}$ is umbilic. From this, we find

\begin{eqnarray}
    \partial_{t} (|N_{t}|^{\frac{1-n}{n}} \int_{N_{t}} H d \mu) &\leq& \frac{1-n}{n} |N_{t}|^{\frac{1-n}{n}} \int_{N_{t}} H d \mu + |N_{t}|^{\frac{1-n}{n}} \int_{N_{t}} \frac{n-1}{n} H d\mu \nonumber \\
    &=& 0.
\end{eqnarray}
Suppose $N_{0}$ admits a long-time, embedded solution to \eqref{IMCF}. $|N_{t}|^{\frac{1-n}{n}}\int_{N_{t}} H d\mu$ is monotone decreasing under IMCF and invariant under scaling $N_{t} \rightarrow \lambda N_{t}$. Evaluating over the $t = +\infty$ limit of $e^{-\frac{t}{n}} N_{t}$ yields

\begin{equation*}
    |N_{t}|^{\frac{1-n}{n}}\int_{N_{t}} H d\mu \geq \frac{n}{r}|\partial B_{r}(0)|^{\frac{1}{n}} = n |\partial B_{1}(0)|^{\frac{1}{n}},
\end{equation*}
and equality is achieved only when $N_{0}$ is umbilic and hence a round sphere.
\end{proof}
Next, we will prove \thref{embeddedness}, which establishes a relationship between solutions of IMCF and minimal disks. This connection is based on Theorem 1 and Corollary 1 from \cite{Meeks1982TheEO} (their original version is formulated for more general domains, so we refine their hypothesis here for our application):

\begin{theorem}[\cite{Meeks1982TheEO}, Theorem 1] \thlabel{embedded_existence}
Let $E_{0} \subset \mathbb{R}^{3}$ be a bounded open domain with smooth, mean-convex boundary, and $\gamma \subset \partial E_{0}$ a Jordan curve. Then there is an embedded disk $D \subset E_{0}$ with boundary $\gamma$ which minimizes area among all immersed disks in $E_{0}$ with the same boundary.
\end{theorem}

\begin{theorem}[\cite{Meeks1982TheEO}, Corollary 1] \thlabel{finiteness}
Let $E_{0} \subset \mathbb{R}^{3}$ be a bounded open domain with smooth, mean-convex boundary, and $\gamma \subset \partial E_{0}$ a Jordan curve of class $C^{4,\alpha}$. For any $k \in \mathbb{R}$, there are only finitely many stable minimally immersed disks in $E_{0}$ with areas less than $k$ that are bounded by $\gamma$.
\end{theorem}
These theorems apply to the immersed minimal disks bounded by $\gamma$ which lie within the mean-convex domain $E_{0}$. Minimal disks bounded by $\gamma$ need not lie in such a domain-- see chapter 5 of \cite{harvie2021thesis} for an example-- meaning that in general the above statements do not hold for all of the minimal disks in $\mathbb{R}^{3}$ of $\gamma \subset \partial E_{0}$. However, in certain special cases, such as when $E_{0}$ is convex, all minimal surfaces bounded on $\partial E_{0}$ lie within it (this is a consequence of the convex hull property for minimal surfaces, see \cite{white2016lectures}). Here, we show that this convexity assumption may be significantly weakened thanks to IMCF and the comparison principle.
\begin{theorem} \thlabel{minimal_comparison}
Let $E \subset \mathbb{R}^{3}$ be a bounded, open set with $\partial E$ a $C^{2}$, $H>0$ connected hypersurface, and $\tilde{E}$ its convex hull. Suppose there exists a family of bounded, open domains $\{ E_{t} \}_{0 \leq t < T}$ in $\mathbb{R}^{3}$ with the following properties.

\begin{enumerate}
    \item $E_{0}=E$ and $\overline{E_{t_{1}}} \subset E_{t_{2}}$ for $t_{1} < t_{2}$.
     \item $\tilde{E} \setminus E \subset \cup_{0 \leq t < T} \partial E_{t}$.
    \item $\partial E_{t}$ is an embedded $C^{2}$ hypersurface with $H>0$ for each $t \in [0,T)$.
\end{enumerate}
Then for any Jordan curve $\gamma \subset \partial E$ and any immersed minimal disk $D$ with boundary $\gamma$, we have $D \subset E$.
\end{theorem}

\begin{proof}
Suppose $D \not \subset E$. Since $D \subset \tilde{E}$, define

\begin{equation*}
    t_{0}= \inf \{ t \in [0,T) | D \subset E_{t}\}.
\end{equation*} 

Property (1) implies that $\overline{E_{t_{0}}} \subset \cap_{t_{0} < t \leq T} E_{t}$, and in fact Property (2) yields equality. Indeed, if $x \in (\cap_{t_{0} < t \leq T} E_{t}) \setminus \overline{E_{t_{0}}}$, then $x \not \in \partial E_{t_{1}}$ for any $t_{1} \in (t_{0},T)$ because $x \in E_{t}$ for $t_{0} < t < t_{1}$ and $E_{t} \cap \partial E_{t_{1}} = \varnothing$ for these $t$. But $x \not \in \partial E_{t_{0}}$ either, contradicting property (2). Since $D \subset E_{t}$ for each $t \in (t_{0},T)$, we have $D \subset \overline{E_{t_{0}}}$.

 Next we claim $D \cap \partial E_{t_{0}} \neq \varnothing$. If $t_{0}=0$ and $D \not \subset E=E_{0}$, then by definition $D \cap \partial E_{t_{0}} \neq \varnothing$. Otherwise, if $D \subset E_{t_{0}}$ one could pick $\delta >0$ small enough so that $D \subset \{ x \in E_{t_{0}} | \text{dist} (x, \partial E_{t_{0}}) > \delta \}$ (such a $\delta$ exists by closedness of $D \cup \gamma$) and $t_{1}<t_{0}$ large enough so that $\{ x \in E_{t_{0}} | \text{dist} (x, \partial E_{t_{0}}) > \delta \} \subset E_{t_{1}}$, again by Property (2). This would contradict the definition of $t_{0}$, so conclude $D \subset \overline{E}_{t_{0}}$ with $D \cap \partial E_{t_{0}} \neq \varnothing$.
 \begin{figure}
 
 \begin{center}
     \textbf{Comparison with $E_{t_{0}}$}
 \end{center}
 \vspace{0.3cm}
     \centering
  \begin{tikzpicture}[xscale=1.8,yscale=1.8]
    \begin{axis}[ y=0.5cm, x=0.5cm,
          xmax=14, xmin=-4, ymax=4.2, ymin=-4.2,
          axis lines=middle,
          restrict y to domain=-7:20,
           xtick = {0},
        ytick= {0},
        yticklabels= {0, $u_{\min}(0)$},
        axis line style={draw=none},
          enlargelimits]
\addplot[domain=-2:0, smooth, thick, name path=f, draw] {(4-x*x)^(0.5)};
\addplot[domain=0:1.5, smooth, thick, name path=g] {0.7*cos(deg(1.047*x + 1.047*6)) +1.3};
\addplot[domain=1.5:4.5, smooth, thick, name path=h] {0.7*cos(deg(1.047*x + 1.047*6)) +1.3};
\addplot[domain=4.5:6, smooth, thick, name path=i] {0.7*cos(deg(1.047*x + 1.047*6)) +1.3};
\addplot[ domain=6:9, smooth, thick, name path=j] {(4 - (x-6)*(x-6))^(0.5)}; 



\draw(3,0) node {$E_{0}$};



\addplot[domain=2.8:3.2, smooth, dashed] {0.6*(1-(25*(x-3)*(x-3)))^(0.5)};
\addplot[domain=2.8:3.2, smooth, dashed] {-0.6*(1-(25*(x-3)*(x-3)))^(0.5)};

\addplot[domain=-2:0, smooth, thick, name path=-f] {-(4-x*x)^(0.5)};
\addplot[domain=0:1.5, smooth, thick, name path=-g] {-0.7*cos(deg(1.047*x + 1.047*6)) -1.3};
\addplot[domain=1.5:4.5, smooth, thick, name path=-h] {-0.7*cos(deg(1.047*x + 1.047*6)) -1.3};
\addplot[domain=4.5:6, smooth, thick, name path=-i] {-0.7*cos(deg(1.047*x + 1.047*6)) -1.3};
\addplot[ domain=6:9, smooth, thick, name path=-j] {-(4 - (x-6)*(x-6))^(0.5)}; 

\addplot [
      thick,
     fill=gray!,opacity=0.4
    ] fill between[of=-f and f];

\addplot [
      thick,
     fill=gray!,opacity=0.4
    ] fill between[of=-g and g];
    
\addplot [
      thick,
     fill=gray!,opacity=0.4
    ] fill between[of=-h and h];
    
\addplot [
      thick,
     fill=gray!,opacity=0.4
    ] fill between[of=-i and i];

\addplot [
      thick,
     fill=gray!,opacity=0.4
    ] fill between[of=-j and j];

\addplot[domain=-3:1.5, smooth, thick, name path=-u] {-(9-x*x)^(0.5)};
\addplot[domain=-3:1.5, smooth, thick, name path=u] {(9-x*x)^(0.5)};
\addplot[domain=4.5:9, smooth, thick, name path=-v] {-(9-(x-6)*(x-6))^(0.5)};
\addplot[domain=4.5:9, smooth, thick, name path=v] {(9-(x-6)*(x-6))^(0.5)};
\addplot[domain=1.5:4.5, smooth, thick, name path=k]{0.2*(x-3)*(x-3) + 2.15};
\addplot[domain=1.5:4.5, smooth, thick, name path=-k]{-0.2*(x-3)*(x-3) - 2.15};

\addplot[domain=-0.4:0.4, smooth, dashed]{3*(1-6.25*x*x)^(0.5)};
\addplot[domain=-0.4:0.4, smooth, dashed]{-3*(1-6.25*x*x)^(0.5)};

\addplot [
      thick,
     fill=gray!,opacity=0.25
    ] fill between[of=u and -u];
    
\addplot [
      thick,
     fill=gray!,opacity=0.25
    ] fill between[of=v and -v];
    
\addplot [
      thick,
     fill=gray!,opacity=0.25
    ] fill between[of=k and -k];

\draw(3,1.5) node{\large{$E_{t_{0}}$}};

\addplot[domain=4:5, smooth, thick, name path=o]{1+ 0.1*(x-4)^(2)};

\addplot[domain=5:6.8, smooth, thick, name path=p]{1+ 0.1*(x-4)^(2)};

\addplot[domain=4:6.8, smooth, thick, dashed]{1.784- 0.1*(x-6.8)^(2)};

\draw(6,1) node{\large{$\gamma$}};

\draw(6.8,2.5) node{\large{\color{blue}{$D$}}};

\addplot[domain=4:5, smooth, name path=n]{2.8- 1.8*(x-5)^(2)};

\addplot[domain=5:6.8, smooth, name path=m]{2.8- 0.3*(x-5)^(2)};

\addplot [
      thick,
     fill=blue,opacity=0.2
    ] fill between[of=o and n];
    
\addplot [
      thick,
     fill=blue,opacity=0.2
    ] fill between[of=p and m];
    
\addplot[domain=3:7, smooth, thick, name path=y]{3.2+ 0.4*(x-5)};

\addplot[domain=3.5:7.5, smooth, thick, name path=z]{2.7+ 0.4*(x-5.5)};

\addplot[domain=3.5:7.5, smooth, thick]{2.7+ 0.4*(x-5.5)};

\addplot[domain=3:3.5, smooth, thick]{2.4- (x-3)};

\addplot[domain=7:7.5, smooth, thick]{4- (x-7)};

\addplot [
      thick,
     fill=red,opacity=0.2
    ] fill between[of=y and z];

\draw(4.8,2.9) node{$x$};

\draw(4.5,3.9) node{\color{red}{$T_{x} D = T_{x} E_{t_{0}}$}};



        \end{axis}
    \end{tikzpicture}
    
  \caption{For a minimal disk $D \not \subset E_{0}$, $H_{D}(x) \geq H_{E_{t_{0}}}(x) >0$ by the comparison principle, yielding a contradiction.}
     \label{exterior_comparison}
 \end{figure}
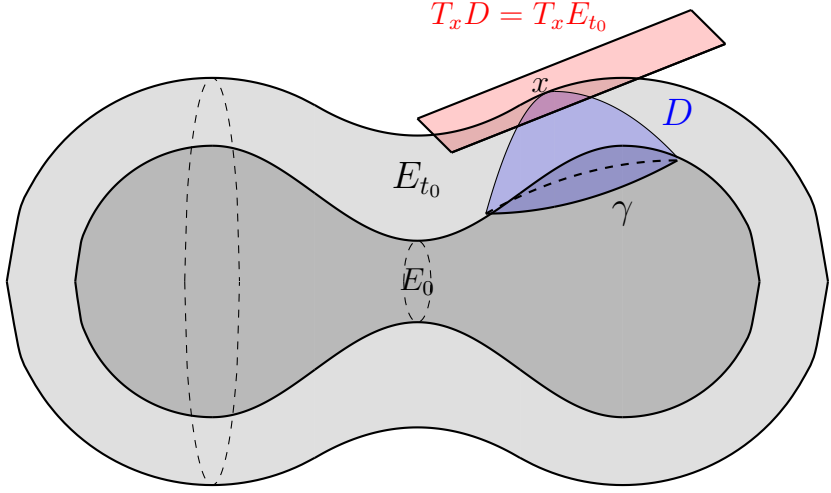
 
 To prove the statement, we utilize a comparison principle. For any $x \in \partial E_{t_{0}} \cap D$ the tangent planes $T_{x}D$ and $T_{x} (\partial E_{t_{0}})$ are parallel, since if not $D \setminus \overline{E_{t_{0}}}$ would be nonempty. Calling the principal curvatures of $\partial E_{t_{0}}$ and $D$ at $x$ $\{\lambda_{i}\}_{1 \leq i \leq n}$ and $\{ \lambda_{i}' \}_{1 \leq i \leq n}$ respectively, we must have
 
 \begin{equation*}
     \lambda_{i}' \geq \lambda_{i}, \hspace{1cm} 1 \leq i \leq n,
 \end{equation*}
 in view of the inclusion $D \subset \overline{E_{t_{0}}}$, see Figure \ref{exterior_comparison}. Property (3) would then yield $H > 0$ at $x \in D$, and this contradicts the minimality assumption. Conclude then that $D \subset E$.
\end{proof}

Suppose that the boundary $N_{0}=\partial E_{0}$ of a smooth, mean-convex domain $E_{0} \subset \mathbb{R}^{n+1}$ admits a solution $\{ N_{t} \}_{0 \leq t < T_{\max}}$ to IMCF that exists forever. The flow surfaces $N_{t}$ may still fail to foliate $\mathbb{R}^{n+1} \setminus E_{0}$, such as in the two spheres example mentioned in Section 3. In that example, though, $N_{t}$ ceases to be embedded in finite time. A solution of IMCF will indeed foliate its image whenever each $N_{t}$ is embedded, and the corresponding domains $E_{t}$ that $N_{t}$ bound allow us to apply \thref{minimal_comparison}.

\begin{corollary} \thlabel{enclosed}
Let $E_{0} \subset \mathbb{R}^{3}$ be a bounded, open domain with  $\partial E_{0}$ be a $C^{\infty}$, $H>0$ connected hypersurface. Suppose the Inverse Mean Curvature evolution $\{ N_{t} \}_{0 \leq t < T}$ of $N_{0}=\partial E_{0}$ satisfies $T_{\max}=+\infty$ and $N_{t}$ is embedded for each $t \in [0,+\infty)$. Then for any Jordan curve $\gamma \subset \overline{E}_{0}$ and any stable immersed minimal disk $D$ with $\partial D= \gamma$, $D \subset E_{0}$.
\end{corollary}
\begin{proof}
According to Theorems 3 and Theorem 4 from \cite{me}, if the flow surfaces $N_{t}=F_{t}(N)$ of the solution to IMCF are embedded, they foliate $U= \mathbb{R}^{n+1} \setminus E_{0}$. Also by these theorems, the domains $E_{t}$ with $N_{t}=\partial E_{t}$ satisfy $\overline{E}_{t_{1}} \subset E_{t_{2}}$ for $t_{1} < t_{2}$ and therefore meet all criteria of \thref{minimal_comparison}. Conclude then that for any immersed minimal surface $D$ with $\partial D \subset \overline{E}_{0}$, we have $D \subset E_{0}$.
\end{proof}

\begin{proof}[Proof of \thref{embeddedness}]
According to the regularity result of \cite{osserman}, the least-area disk $D$ spanning any Jordan curve $\gamma \subset \mathbb{R}^{3}$ is immersed, and so \thref{enclosed} implies that $D \subset E_{0}$. This $D$ also minimizes area within $E_{0}$, so it must correspond to the least-area immersion in $E_{0}$ guaranteed by \thref{embedded_existence} and is therefore embedded. Furthermore, all immersed minimal disks bounded by $\gamma$ are contained within $E_{0}$, and so the second part of the theorem follows from Corollary 1 of \cite{Meeks1982TheEO}.
\end{proof}
\appendix
\section{Appendices}

\subsection{Non-Cylindrical Maximum Principle}
We recall \thref{noncyl}:

\begin{theorem}
Let $F: N^{n} \times [0,T) \rightarrow \mathbb{R}^{n+1}$ be a solution of the Inverse Mean Curvature Flow \eqref{IMCF} over a closed manifold $N$. For a domain $U \subset N \times [0,T)$ and $f \in C^{2,1}(U) \cap C(\overline{U})$, suppose for a smooth vector field $\eta$ over $U$ we have

\begin{equation*}
    (\partial_{t} - \frac{1}{H^{2}} \Delta) f \leq \langle \eta, \nabla f \rangle.
\end{equation*}
(Resp. $\geq$ at a minimum) Here $\Delta$ and $\nabla$ are the Laplacian and gradient operators over $N_{t}$, respectively. Then

\begin{equation*}
    \sup_{U} f \leq \sup_{\partial_{P} U} f
\end{equation*}
(Resp. $\inf_{U} f \geq \inf_{\partial_{P} U} f$). Furthermore, suppose that $f$ has a positive supremum over $U$ and that each $(x_{0},t_{0}) \in \partial_{P} U \setminus \tilde{\partial}_P U$ is a limit point of $U \cap \{ t < t_{0} \}$. Then
\begin{equation} \label{reduced}
    \sup_{U} f \leq \sup_{\tilde{\partial}_{P} U} f
\end{equation}
(Resp. $\inf_{U} f \geq \inf_{\tilde{\partial}_{P} U} f$ for a positive minimum).
\end{theorem}

\textit{ Proof of \thref{noncyl}}: For the first part, we follow the proof in the Appendix of \cite{maria3}. For a given smooth vector field $\eta$ over $U$ we have by hypothesis

\begin{equation*}
    (\partial_{t} - \frac{1}{H^{2}} \Delta) f \leq \langle \eta, \nabla f \rangle.
\end{equation*}

We argue by contradiction: define the function $\tilde{f}(x,t)=f(x,t)- \epsilon t$ for some $\epsilon >0$. Then
\begin{equation*}
    \partial_{t} \tilde{f}= \partial_{t} f- \epsilon, \hspace{0.5cm}
    \partial_{i} \tilde{f} = \partial_{i} f, \hspace{0.5cm} \partial_{ij} \tilde{f} = \partial_{ij} f.
\end{equation*}

The operator over $\tilde{f}$ must then obey

\begin{equation} \label{principle}
    (\partial_{t} - \frac{1}{H^{2}} \Delta - \eta \cdot \nabla) \tilde{f} < 0.
\end{equation}
On the other hand, at any interior maximum $(x_{0},t_{0}) \in U$ of $\tilde{f}$, the criteria for a local maximum dictate that at $(x_{0},t_{0})$

\begin{equation*}
    \partial_{t} \tilde{f} \geq 0, \hspace{0.5cm} \partial_{i} \tilde{f}=0, \hspace{0.5cm} \partial_{ij} \tilde{f} \leq 0,
\end{equation*}
where the last inequality is in the operator-theoretic sense for the symmetric matrix $\partial_{ij} \tilde{f}$. Writing 

\begin{equation*}
    \Delta \tilde{f}= g^{ij} (\partial_{ij} \tilde{f} - \Gamma^{k}_{ij} \partial_{k} \tilde{f}), \hspace{0.5cm} \nabla \tilde{f} = g^{ij} \partial_{j} \tilde{f} \partial_{i} \vec{F}
\end{equation*}
in view of the positivity of $g_{ij}$, we see $\Delta \tilde{f} \leq 0$ and $\nabla f=0$. Hence 

\begin{equation*}
    (\partial_{t} - \frac{1}{H^{2}} \Delta - \eta \cdot \nabla) \tilde{f}(x_{0},t_{0}) \geq 0,
\end{equation*}
contradicting \eqref{principle}. So $\tilde{f}$ has no interior maximum and thus
\begin{equation*}
   \sup_{U} f - \epsilon T \leq  \sup_{U} \tilde{f} \leq \sup_{\partial_{P} U} f.
\end{equation*}
Then $\sup_{U} f \leq \sup_{\partial_{P} U} f + \epsilon T$. For $T< \infty$, letting $\epsilon \rightarrow 0$ yields the result. To prove the statement for the infimum, take $\tilde{f}=f- \epsilon t$, $\epsilon >0$, and repeat this argument for a minimum.

For the second part, we show $\sup_{U} f \leq \sup_{\tilde{\partial}_{P} U} f$ if $\sup_{U} f > 0$. Define $Z= \partial_{P} U \setminus \tilde{\partial}_{P} U$, and for each $t \in [0,T)$ let $Z_{t}= Z \cap \{ t \}$ be the cross sections of $Z$. We argue by contradiction: suppose \eqref{reduced} does not hold. Then the maximum of $f$ does not occur on $\partial_{\tilde{P}} U$ nor does it occur at an interior point of $U$, so it must occur on the set $Z$. Call $Z_{\max}$ the union of $Z_{t}$'s on which the maximum is achieved, and let $t_{*}>0$ be the first time at which $\sup_{U} f$ is achieved on $Z_{\max}$. Pick $\beta > 0$ such that

\begin{equation*}
    \sup_{U} f > \beta >  \sup_{\partial_{P} U \setminus Z_{\max}} f,
\end{equation*}
and define

\begin{equation*}
    Y_{\beta} = \{ (x,t) \in U| f(x,t) > \beta \}.
\end{equation*}

$\overline{Y_{\beta}}$ must intersect $Z_{t_{*}}$. Consider $(x_{0},t_{*}) \in Z_{t_{*}} \cap \overline{Y_{\beta}}$. From the additional assumption in the proposition there is a sequence of points $(x_{n},t_{n}) \in U$ with $t_{n} < t_{*}$ converging to $(x_{0},t_{*})$. By continuity of $f$, $f(x_{n},t_{n}) > \beta$ for large enough $n$. This means that the set

\begin{equation}
    X_{\beta}=\{ t < t_{*} \} \cap Y_{\beta}
\end{equation}
is nonempty and open. By openness, we pick a time $t_{1} < t_{*}$ so that the set $X_{\beta} \cap \{ t \leq t_{1} \} \neq \varnothing$. 
\begin{figure}
\begin{center}
    \textbf{The Cutoff in Time}
    \vspace{0.3cm}
\end{center}
\centering
\begin{tikzpicture}[ xscale=1.7,yscale=1.7,xshift=-2cm]
\begin{axis}     [ 
ylabel={\Large{$t$}},
      xlabel={\Large{$N$}},
          xmax=5, xmin=-1, ymax=1.4, ymin=0.3,
          restrict y to domain=-2:1.8,
           ytick = {0.2, 0.45, 0.55, 0.625, 1.5},
       yticklabels = {$0$, $t_{1}$, $t_{2}$, $t^{*}$ ,$T$},
        xtick= {-3},
          enlargelimits]

\draw(2,0.9) node{\Large{$U$}};

\draw(-0.15,0.625) node{\Large{$\cdot$}};

\draw (-0.15,0.5) node{$X_{\beta}$};

\addplot [domain= -0.5:0.2, smooth, dashed, name path=y] {1.8*(x+0.15)^(2) + 0.4};

\addplot [domain= -0.3:0, smooth, dashed, name path=w ]{0.45};
\addplot [domain= -1:4.5, smooth, opacity=0, name path=u ]{0.55};
\addplot [domain= -1:4.5, smooth, name path=-u ]{0};

\addplot [domain= -0.5:0.2, smooth, name path=z] {0.625};




\draw (-0.1,0.625) node[anchor=south]{\small{$Z_{t_{*}}$}};
\addplot [domain=-1:-0.5, smooth, thick, name path=f] {2.5*(x+1)^(2)};
\addplot [domain=-0.5:0.25, smooth, dashed, name path=g] {0.625};
\addplot [domain=0.25:0.8, smooth, thick, name path=h] {3*(x-0.25)^(2)+0.625};
\addplot [domain=0.8:1.8, smooth, thick, name path=i] {1.5325};
\addplot [domain=1.8:2.2, smooth, thick, name path=j] {-2*(x-1.8)^(2)+1.5325};
\addplot [domain=2.2:3, smooth, dashed, name path=k] {1.2125};
\addplot[domain=3:3.3, smooth, thick, name path=l] {-(x-3)+1.2125};

\draw (2.8,1.2125) node[anchor=south]{$\partial_{P} U \setminus \tilde{\partial_{P}} U$};

\addplot[domain=3.3:3.8, smooth, dashed, name path=m] {0.9125};
\addplot[domain=3.8:4.5, smooth, thick, name path=n] {1.091*(4.5-x)^(0.5)};

\addplot[domain=-1:-0.5, smooth, thick, name path=-f] {0};
\addplot[domain=-0.5:0.25, smooth, thick, name path=-g] {0};
\addplot[domain=0.25:0.8, smooth, thick, name path=-h] {0};
\addplot [domain=0.8:1.8, smooth, thick, name path=-i] {0};
\addplot [domain=1.8:2.2, smooth, thick, name path=-j] {0};
\addplot [domain=2.2:3, smooth, thick, name path=-k] {0};
\addplot[domain=3:3.3, smooth, thick, name path=-l] {0};
\addplot[domain=3.3:3.8, smooth, thick, name path=-m] {0};
\addplot[domain=3.8:4.5, smooth, thick, name path=-n] {0};

\addplot [
      thick,
     fill=gray!,opacity=0.4
    ] fill between[of=-f and f];

\addplot [
      thick,
     fill=gray!,opacity=0.4
    ] fill between[of=-g and g];

\addplot [
      thick,
     fill=gray!,opacity=0.4
    ] fill between[of=-h and h];

\addplot [
      thick,
     fill=gray!,opacity=0.4
    ] fill between[of=-i and i];

\addplot [
      thick,
     fill=gray!,opacity=0.4
    ] fill between[of=-j and j];

\addplot [
      thick,
     fill=gray!,opacity=0.4
    ] fill between[of=-k and k];

\addplot [
      thick,
     fill=gray!,opacity=0.4
    ] fill between[of=-l and l];

\addplot [
      thick,
     fill=gray!,opacity=0.4
    ] fill between[of=-m and m];

\addplot [
      thick,
     fill=gray!,opacity=0.4
    ] fill between[of=-n and n];

\addplot [
      thick,
     fill=white,opacity=0.5
    ] fill between[of=-u and u];

\draw[->] (0.5,0.45) -> (0.1,0.45);
\draw (0.5,0.45) node[anchor=west]{$(\phi f) > \beta > \sup_{\partial_{P} U \setminus Z_{\max}} f$};

\draw (2,0.7) node{$\phi f \equiv 0$};
\end{axis}
    \end{tikzpicture}
\caption{The cutoff function $\phi$ is $1$ for times less than $t_{1}$ and $0$ for times greater than $t_{2}$. This guarantees that the supremum of $\phi f$ occurs at an interior point of $X_{\beta}'$.}
    \label{cutoff}
\end{figure}
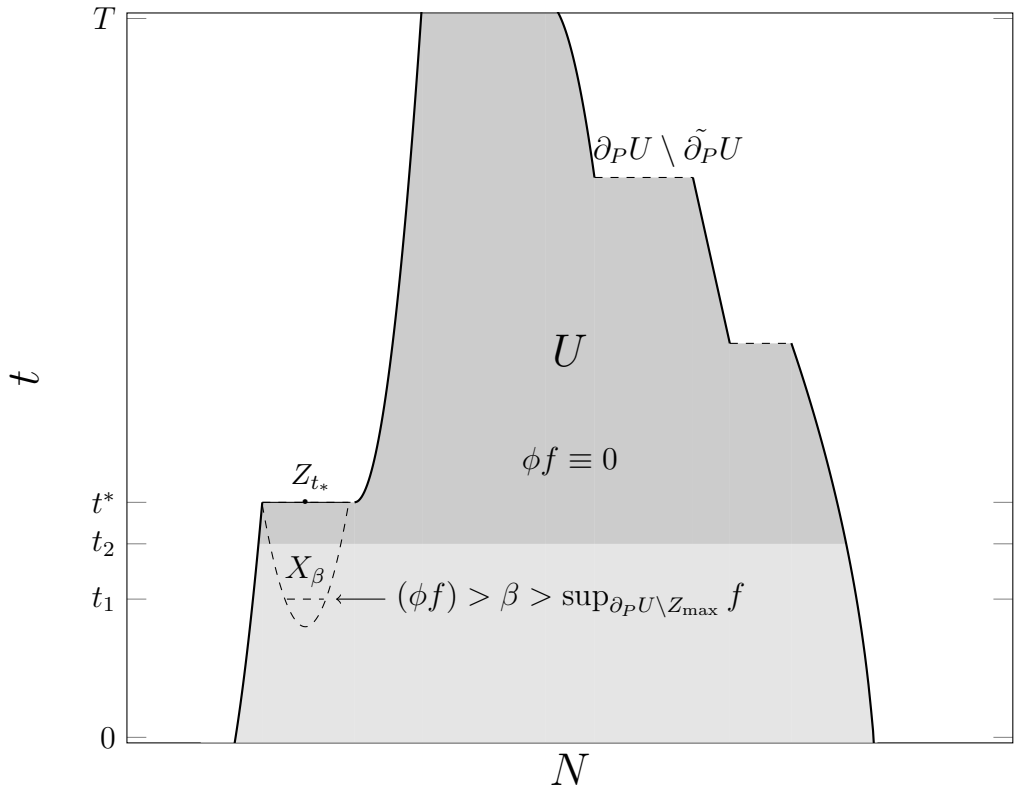

Fix a time $t_{2} \in (t_{1},t_{*})$ and choose a cutoff function $\phi: [0,T] \rightarrow [0,1]$ such that $\phi(t)=1$ when $t \in [0,t_{1}]$, $\phi'(t) < 0$ when $t \in (t_{1}, t_{2})$, and $\phi(t)=0$ when $t \in [t_{2},T]$, see Figure \ref{cutoff}. Since $(\phi f) (x, t_{1})= f(x, t_{1}) > \beta$ for $(x,t_{1}) \in X_{\beta}$, we know $\sup_{U} \phi f > \beta >0$. 

We calculate

\begin{equation*}
    \partial_{t} (\phi f) = \phi \partial_{t} f + \phi' f, \hspace{0.5cm} \Delta \phi f = \phi \Delta f, \hspace{0.5cm} \nabla (\phi f) = \phi \nabla f,
\end{equation*}

so 

\begin{equation} \label{phi*f}
    (\partial_{t} - \frac{1}{H^{2}} \Delta - \eta \cdot \nabla) (\phi f) = \phi (\partial_{t} - \frac{1}{H^{2}} \Delta - \eta \cdot \nabla) f + \phi' f.
\end{equation}

By hypothesis we have $(\partial_{t} - \frac{1}{H^{2}} \Delta - \eta \cdot \nabla) f (x,t) \leq 0$ and $\phi' \leq 0$. Since $\sup_{U} \phi f > 0 $, any interior point $(x_{0},t_{0}) \in U$ at which $\phi f$ achieves this supremum would need to satisfy

\begin{equation*}
    (\partial_{t} - \frac{1}{H^{2}} \Delta - \eta \cdot \nabla) (\phi f) (x_{0},t_{0}) < 0.
\end{equation*}

By the same argument used for the first part of \thref{noncyl}, this is impossible, and so $\sup_{U} (\phi f) = \sup_{\partial_{P} U} (\phi f)$. However, $\sup_{\partial_{P} U \setminus Z_{\max}} \phi f \leq \sup_{\partial_{P} U \setminus Z_{\max}} f < \beta$ by hypothesis, and $\phi f|_{Z_{\max}} \equiv 0$ as $\phi =0$ on $[t_{*},T)$. This would altogether yield

\begin{equation*}
    \sup_{\partial_{P} U} \phi f < \beta < \sup_{U} \phi f, 
\end{equation*}
contradicting the first part of the non-cylindrical maximum principle. Conclude then that

\begin{equation*}
    f(x,t) \leq \sup_{\tilde{\partial}_{P}U} f
\end{equation*}
on $U$. The statement may be shown for a minimum by choosing $\beta>0$ with $\inf_{U} f < \beta < \inf_{\partial_{P} U \setminus Z_{\max}} f$.

\begin{remark}
The version of this principle used in \cite{maria3} and \cite{head2019singularity} does not include the hypothesis that $U$ approaches $Z_{t_{*}}$ from below in time. However, if $U$ only touches $Z_{t_{*}}$ from above in time, $Y_{\beta} \cap \{ t < t_{*}\}$ may be empty. The corresponding cutoff function would then need to be chosen to increase with $t$, so that the last term in \eqref{phi*f} is possibly non-negative. Therefore, this additional hypothesis seems to be necessary.
\end{remark}
\subsection{Non-Star-Shaped Admissible Initial Data}
\begin{proposition} \label{non_star_shaped}
For any $n \geq 2$, there is an admissible surface $N_{0}^{n} \subset \mathbb{R}^{n+1}$ which is not star-shaped.
\end{proposition}
\begin{proof}
We will begin with an example of a non-star-shaped admissible surface which is of class $C^{2}$. Define the domain $E_{0} \subset \mathbb{R}^{n+1}$ by

\begin{equation}
    E_{0}= \{ -10 < x_{1} < 10 | (x_{2}^{2} + \dots x_{n+1}^{2})^{\frac{1}{2}} < y(x_{1}) \},
\end{equation}
where $y: (-10,10) \rightarrow \mathbb{R}^{+}$ is the even extension of the function

\begin{equation} \label{y(x)}
    y(x) = \begin{cases} \frac{3}{4} & x \in [0,8) \\
    \frac{3}{4} + 2 (x-8)^{3} - \frac{11}{4}(x-8)^{4} + (x-8)^{5} & x \in [8,9) \\
    \sqrt{1-(x-9)^{2}} & x \in [9,10),
    \end{cases}
\end{equation}
across $x=0$. Direct computation shows that $y \in C^{2}((-10,10))$. One can also verify through either computing maxima or graphing that

\begin{eqnarray}
   \frac{3}{4} \leq y(x) (1 + |y'(x)|^{2})^{\frac{1}{2}} &\leq& \frac{5}{4} \label{est1} \\
    \frac{y''(x)y(x)}{1+|y'(x)|^{2}} &<& 1 \label{est2},
\end{eqnarray}
for each $x \in (-10,10)$. The left-hand side of \eqref{est2} corresponds to the ratio $-\frac{k}{p}$ of the principal curvatures of $N_{0}$, and so this bound implies

\begin{equation} \label{H_pos}
    \min_{N_{0}} H >0,
\end{equation}
The function in \eqref{est1} corresponds to $p^{-1}$ on $N_{0}$, meaning
\begin{equation} \label{p_good}
 \frac{\max_{N_{0}} p}{\min_{N_{0}} p} < 2 \leq n^{\frac{n}{2(n-1)}}, \hspace{0.5cm} n \geq 2,
\end{equation}
so altogether $N_{0}$ is mean-convex and admissible. 

To demonstrate that $N_{0}$ is not star-shaped, it is sufficient to consider graph of the function $y$ in the $x-y$ plane. In this plane, the line segment connecting $(0,0)$ to $(9,1)$, parametrized by $L(x)=\frac{1}{9}x$, satisfies

\begin{equation*}
    L(8) = \frac{8}{9} > \frac{3}{4}= y(8),
\end{equation*}
meaning this segment must intersect the graph of $y$ at another point by the intermediate value property. Furthermore, for any $x_{0} \leq 0$ and $0 \leq y_{0} < y(x_{0})$, the line segment connecting $(x_{0},y_{0})$ to $(9,1)$ satisfies 

\begin{equation*}
    L_{(x_{0},y_{0})} (x) \geq L(x), \hspace{0.3cm} x \in [0,9],
\end{equation*}
for the $L(x)$ defined above. So $L_{(x_{0},y_{0})}(8) > \frac{3}{4}$, and once again by the intermediate value property this segment intersects the graph somewhere else. If we choose $x_{0} \geq 0$, $0 \leq y_{0} < y(x_{0})$, the segment connecting $(x_{0},y_{0})$ to $(-9,1)$ must also intersect since $y$ is even.

\begin{figure}
    \centering
    
    \textbf{A Non-Star-Shaped $N_{0}$: The Generating Graph}
    \vspace{0.3cm}
    \begin{center}
    \begin{tikzpicture}[xscale=1.4,yscale=1.4]
    \begin{axis}[ylabel=\footnotesize{$y(x)$}, xlabel=x, y=0.5cm, x=0.5cm,
          xmax=10.5, xmin=-10.5, ymax=2, ymin=-0.25,
          axis lines=middle,
          restrict y to domain=-10:10, xtick={-10,10},ytick={0, 0.75}, yticklabels={0,$\frac{3}{4}$}]
        \addplot[name path=f, domain = -8:8, smooth]{(0.75)};
        \addplot[name path=o, domain=-8:8, smooth]{(0)};
         \addplot [
        thick,
        color=gray,
        fill=gray, 
        fill opacity=0.4
    ]fill between[of=f and o];
        \addplot[name path=g, domain = 8:9, smooth]{0.75 + 2*(x-8)^(3) -2.75* (x-8)^(4) + (x-8)^(5)};
        \addplot[name path=go, domain= 8:9, smooth]{0};
           \addplot [
        thick,
        color=gray,
        fill=gray, 
        fill opacity=0.4
    ]fill between[of=g and go];
        \addplot[name path=h, domain = -10:-9, smooth]{(1-(x+9)^(2))^(0.5)};
        \addplot[name path=ho, domain= -10:-9, smooth]{0};
           \addplot [
        thick,
        color=gray,
        fill=gray, 
        fill opacity=0.4
    ]fill between[of=h and ho];
         \addplot[name path=i, domain = 9:10, smooth]{(1-(x-9)^(2))^(0.5)};
         \addplot[ name path=io, domain= 9:10, smooth]{0};
            \addplot [
        thick,
        color=gray,
        fill=gray, 
        fill opacity=0.4
    ]fill between[of=i and io];
         \addplot[name path=j, domain = -9:-8, smooth]{0.75 - 2*(x+8)^(3) -2.75* (x+8)^(4) - (x+8)^(5)};
         \addplot[name path=jo, domain= -9: -8, smooth]{0};
            \addplot [
        thick,
        color=gray,
        fill=gray, 
        fill opacity=0.4
    ]fill between[of=j and jo];
        \addplot[domain=-3:9, smooth, color=blue] {0.0625*(x+3) + 0.25};
        \draw[color=blue] (-3,0.25) node[anchor=east]{\small{\color{blue}{$(x_{0},y_{0})$}}};
        \draw[color=blue] (9,1) node[anchor=south]{\small{\color{blue}{$(9,1)$}}};
    \end{axis}
      \end{tikzpicture}
    \end{center}
    \caption{A graph of the function $y(x)$ defined in \eqref{y(x)}. The admissible surface generated by $y(x)$ is not star-shaped because of its long neck.}
    \label{fig:my_label}
\end{figure}

For any $x_{0} \in E_{0}$, let $P \subset \mathbb{R}^{n+1}$ be the $2$-plane containing $x_{0}$ and $X_{1}$ (if $x_{0}$ lies on this axis, let $P$ be any $2$-plane containing $X_{1}$). Choose Cartesian coordinates $(x,y)$ on $P$ so that $x$-axis is $X_{1}$ and the $y$-coordinate of $x_{0}$ is non-negative. Then either the line segment connecting $x_{0}$ to $(9,1) \in P$ or the line segment connecting $x_{0}$ to $(-9,1) \in P$ is not contained in $E_{0}$. Since these points lie on $N_{0}$, $N_{0}$ cannot be star-shaped with respect to $x_{0}$.

To conclude, we must show that $N_{0}$ is approximated in a $C^{2}$ sense by smooth rotationally symmetric hypersurfaces. For $P$ as above, $N_{0} \cap P= \gamma$ is a closed plane curve parametrized over $\sfrac{[0,1]}{\mathbb{Z}} \simeq \mathbb{S}^{1}$ by

\begin{equation}
    \gamma(t)= (x_{1}(t),x_{2}(t))= \begin{cases}
        (10(4t-1), y(10(4t-1))) & 0 \leq t \leq \frac{1}{2},  \\
        (10(3-4t), -y(10(3-4t)) & \frac{1}{2} \leq t \leq 1.
    \end{cases}
\end{equation}
By convolving with a symmetric mollifier $\psi_{\epsilon}$ over $\mathbb{S}^{1}$, we obtain smooth functions $x^{\epsilon}_{1}(t)$ and $x^{\epsilon}_{2}(t)$ on $\mathbb{S}^{1}$ which are repsectively even and odd about $t=\frac{1}{2}$. For $\epsilon$ small enough, the corresponding smooth curve $\gamma^{\epsilon}(t)=(x^{\epsilon}_{1}(t),x^{\epsilon}_{2}(t))$ in $P$ is uniformally convex over small intervals containing $t=0$ and $t=\frac{1}{2}$, and away from these intervals $x^{\epsilon}_{1}(t)$ and $x^{\epsilon}_{2}(t)$ converge in uniform $C^{2}$ topology to $x_{1}(t)$ and $x_{2}(t)$ as $\epsilon \rightarrow 0$. Then $\gamma_{\epsilon}$ is the cross section $N^{\epsilon}_{0} \cap P$ of a smooth, rotationally symmetric hypersurface $N^{\epsilon}_{0}$ which inherits properties \eqref{H_pos} and \eqref{p_good} and fails to be star-shaped. 
\end{proof}

\printbibliography[title={References}]

\begin{center}
\textnormal{ \large National Center for Theoretical Sciences, Mathematics Division \\
National Taiwan University \\
Taipei City, Taiwan 10617\\
e-mail: bharvie@ncts.ntu.edu.tw}\\
\end{center}
\end{document}